\documentclass[12pt]{article}
\usepackage{amsmath,amssymb,indentfirst,bm}
\usepackage{graphicx, epsfig, url}
\usepackage[authoryear]{natbib}
\usepackage{amsthm, verbatim}
\usepackage{setspace}
\usepackage{xcolor}
\usepackage{enumitem}

\newtheorem{definition}{Definition}[section]

\newtheorem{example}{Example}[section]
\newtheorem{lemma}{Lemma}[section]

\newtheorem{corollary}{Corollary}[section]
\newtheorem{proposition}{Proposition}[section]
\newtheorem{criterion}{Criterion}[section]
\newtheorem{counterexample}{Counterexample}[section]
\newtheorem{method}{Method}[section]

\usepackage[
  a4paper,
  left=25mm,
  right=25mm,
  top=25mm,
  bottom=25mm,
  includeheadfoot
]{geometry}

\title{Version-Robust Methods for Identifying Minimal Sufficient Statistics}

\author{Rafael Oliveira Cavalcante\\
     {\it \footnotesize Departamento de Estat\'istica, IME,
     Universidade de S\~ao Paulo}
     \vspace{-0.2cm}\\
     {\it \footnotesize Rua do Mat\~ao, 1010, S\~ao Paulo/SP, 05508-090, Brazil}
     \vspace{-0.2cm}\\
     {\footnotesize email: {\tt rflolcante@gmail.com}}
     \and
     Alexandre Galv\~ao Patriota\\
     {\it \footnotesize Departamento de Estat\'istica, IME,
     Universidade de S\~ao Paulo}
     \vspace{-0.2cm}\\
     {\it \footnotesize Rua do Mat\~ao, 1010, S\~ao Paulo/SP, 05508-090, Brazil}
     \vspace{-0.2cm}\\
     {\footnotesize email: {\tt patriota@ime.usp.br}}
}
\date{}

\begin{document}
\maketitle

\onehalfspacing
\begin{abstract}
Let $f_\theta$ be the joint density of a random sample $X$. A frequently used criterion asserts that a statistic $T(X)$ is minimal sufficient if, for any sample points $x$ and $y$, $T(x) = T(y)$ exactly when there exists a finite constant $h_{xy} > 0$, independent of $\theta$, such that $f_\theta(y) = f_\theta(x)h_{xy}$ for all $\theta$. We show that this criterion is false in general via a counterexample exploiting the non-uniqueness of versions of Radon--Nikodym derivatives. Although \cite{Sato1996} established sufficient regularity conditions for the validity of this criterion, these conditions are frequently intractable to verify in practice. We resolve this limitation by introducing a version-robust method applicable whenever sufficiency is known. Moreover, our method allows us to generalize Sato's approach from Euclidean settings to arbitrary analytic Borel sample spaces and separable measurable statistic spaces. We also obtain a method for exponential-family densities under verifiable hypotheses. Taken together, these results clarify when pointwise likelihood-ratio arguments for minimal sufficiency are mathematically sound in irregular settings. Finally, we construct a counterexample demonstrating that a distinct criterion for minimal sufficiency due to \cite{Pfanzagl1994, Pfanzagl2017} similarly fails in the absence of supplementary hypotheses.
 Identifying minimal sufficient statistics is important not only for parsimonious data reduction but also because, in models admitting complete sufficiency, such statistics provide a practical route to the complete sufficient structure underlying optimal estimation and prediction.
\end{abstract}


\section{Introduction}
\label{sec:intro}

Minimal sufficiency is a data-reduction problem in dominated statistical models: one seeks a statistic that retains all inferentially relevant information in the sample while being minimal among sufficient reductions. A common strategy is to compare sample points through proportionality relations between likelihoods or densities. According to the Lehmann--Scheff\'e theorem \citep[Theorem 7.17]{LieseMiescke2008}, one route to finding a uniformly minimum-variance unbiased estimator is to condition an unbiased estimator on a statistic that is both sufficient and complete. In practice, however, complete sufficient statistics are often difficult to construct directly. Nevertheless, the relation between completeness and minimal sufficiency remains important: whenever a model admits a sufficient and complete statistic, every minimal sufficient statistic is automatically complete. For this reason, methods for identifying minimal sufficient statistics are useful not only for data reduction itself but also, in such models, as an indirect route to complete sufficiency in applications such as estimation theory. This connection is also relevant for prediction, since in such models an identified minimal sufficient statistic may serve as the complete sufficient statistic underlying the Bayes-optimal prediction procedures with frequentist coverage control developed in \cite{Hoff2023}.

Let $\Theta$ be the parameter space and $T$ be a statistic defined on the statistical model. Informally, $T$ is said to be sufficient if, and only if, the data distribution conditional on $T$ does not depend on $\theta\in \Theta$, and $T$ is said to be minimal sufficient if, and only if, it is sufficient and, given any sufficient statistic $S$, there exists a measurable map $f$ such that $T = f \circ S$ almost everywhere. Throughout this paper, the formal definitions of statistical models, statistics, sufficiency and minimal sufficiency, standard and analytic Borel spaces, countably generated spaces and separable measurable spaces are collected in Appendix~\ref{sec:definitions}.

In this paper, we discuss the limitations of two pointwise criteria for identifying minimal sufficiency, propose version-robust methods that avoid such limitations, and also generalize a method proposed by \cite{Sato1996}. The first and widely cited criterion considered in this paper can be found in texts such as \cite{Schervish1995}, see his Theorem 2.29, \cite{Wasserman2004}, see his Theorem 9.36, and \cite{Mavrakakis2021}, see their Proposition 10.1.13, and it is stated as follows.
\begin{criterion}
\label{crit:2.1}
Let $f_\theta (x)$ be a joint density of a random sample $X$. A statistic $T(X)$ is minimal sufficient if, for any two sample points $x$ and $y$, we have $T(x)=T(y)$ iff there exists a finite constant $h_{xy}>0$ independent of $\theta$ such that \[f_\theta (y)=f_\theta (x)h_{xy} \quad \mbox{for every} \ \theta\in \Theta.\]
\end{criterion}

In this formulation, the restriction applies over the entire (possibly uncountable) parameter space, which makes the pointwise proportionality relation sensitive to the chosen versions of the densities and allows us to develop Counterexample \ref{ex:3.1}. 
Although widely quoted as a method for minimal sufficiency \citep[see also][]{Young2005,Makarov2013,Olive2014,Gasperoni2025}, 
Criterion~\ref{crit:2.1} is often stated without explicit regularity hypotheses and is false in full generality. In particular, since densities are only defined almost everywhere, one can modify versions on null sets in a $\theta$-dependent way and thereby change the pointwise proportionality relation. 

The formal conditions under which a reformulation of this method holds, preventing our counterexample, were first established by \cite{Sato1996}. This method is restricted to Euclidean spaces and is not simple to apply in practice. To bypass this practical difficulty, when a candidate statistic is already known to be sufficient (a condition that is typically easy to verify via the Neyman--Fisher factorization theorem), we introduce a more direct version-robust minimality method (Method~\ref{meth:3.1}). We then use this approach to extend Sato's method beyond Euclidean spaces to analytic Borel sample spaces and separable measurable statistic spaces (Method~\ref{meth:3.2}), and to establish a minimal-sufficiency method for exponential-family densities (Method~\ref{meth:3.3}).

The second criterion considered in this paper was developed by \cite{Pfanzagl1994} and reformulated in \cite{Pfanzagl2017}. This approach is formally stated in terms of dominated statistical models, countably generated measurable spaces, measurable functions and standard Borel spaces, which are revisited in our Appendix~\ref{sec:definitions}. By using these concepts, we present Pfanzagl's criterion.  

\begin{criterion}
\label{crit:2.6}
Let $(\mathcal{X},\Sigma,\{P_\theta\}_{\theta\in\Theta })$ be a statistical model and $\mu :\Sigma\to\overline{\mathbb{R}}$ a $\sigma$-finite measure such that $P_\theta \ll \mu $ for every $\theta\in\Theta$ (i.e., $P_\theta$ is absolutely continuous w.r.t. $\mu $ for every $\theta\in \Theta$). Suppose that $(\mathcal{X},\Sigma )$ is countably generated. Let $(\mathcal{T},\Sigma _\mathcal{T})$ be a standard Borel space and $T:\mathcal{X}\to \mathcal{T}$ a measurable function. For each $\theta\in\Theta$, let $f_\theta :\mathcal{X}\to \mathbb{R}$ be a density of $P_\theta$ w.r.t. $\mu$. Suppose that, for each $\theta\in\Theta$, we have $f_\theta =g_\theta (T)h$, where $h:\mathcal{X}\to \mathbb{R}$ and $g_\theta :\mathcal{T}\to \mathbb{R}$ are two non-negative measurable functions. If there exists a non-empty countable subset $\Theta _0\subseteq\Theta$ such that, for any $t_1,t_2\in \mathcal{T}$ satisfying $g_\theta (t_1)=g_\theta (t_2)$ for every $\theta\in\Theta _0$, we have $t_1=t_2$, then $T$ is minimal sufficient.
\end{criterion}
This criterion corresponds to Theorem 1.4.4 of \cite{Pfanzagl1994}. A later version of this result in \cite{Pfanzagl2017} is stated with the additional requirement that $\{P_{\theta }\}_{\theta\in\Theta _0}$ is dense in $\{P_\theta \}_{\theta\in \Theta }$ w.r.t. the total variation distance, defined as $d(P,Q):=\sup _{E\in \Sigma }|P(E)-Q(E)|$. However, as noted at the beginning of the proof in \cite{Pfanzagl1994}, this density assumption is not necessary and may be omitted. Note also that Pfanzagl's criterion imposes a restriction only on at most a countable subset of the parameter space. However, it still does not hold without additional assumptions, as shown by our Counterexample~\ref{ex:3.2}. 

The remainder of this paper is organized as follows. In Section \ref{sec:counterexamples}, the limitations of these methods are demonstrated through specific counterexamples that motivate the need for version-robust methods. Section \ref{sec:methods} introduces the corrected and generalized methods that form the core contribution of this work, followed by examples illustrating their applications. Section \ref{sec:proofs} is dedicated to the formal proofs of these methods.  Appendix \ref{sec:definitions} presents the measurable-space background used throughout the paper, namely, statistical models (Definition \ref{def:statistical_model}), sufficient statistics (Definition \ref{def:suff}), standard Borel spaces (Definition \ref{def:stand_Borel}), countably generated spaces (Definition \ref{def:countable_space}), separable measurable spaces (Definition \ref{def:separable}), and analytic Borel spaces (Definition \ref{def:analytic_Borel}). We denote the set of natural numbers by $\mathbb{N} = \{0,1,2, \ldots \}$. Moreover, we will abbreviate the expressions ``if, and only if,'' and ``with respect to'' by ``iff'' and ``w.r.t.'', respectively.

\section{Counterexamples}
\label{sec:counterexamples}

In this section, we construct counterexamples showing Criteria \ref{crit:2.1} and \ref{crit:2.6} require further regularity hypotheses. The first, similar to the example from \citet{Barndorff1976} and closely related to Example 1 of \cite{Taraldsen2026}, perturbs the Gaussian density at a single $\theta$-dependent point, falsifying Criterion \ref{crit:2.1}. The second invalidates Criterion \ref{crit:2.6} on a finite probability space.
These counterexamples motivate the methods introduced in Section \ref{sec:methods}.

\begin{counterexample}
\label{ex:3.1}
Let $X_1,\cdots, X_n$ be a random sample with normal distribution $N(\theta, 1)$, $\theta\in\mathbb{R}$, with $n\geq 2$. Using the Neyman-Fisher Factorization Theorem, we can conclude that $\sum _{i=1}^nX_i$ is sufficient. Denote by $\lambda ^n$ the Lebesgue measure on $\mathbb{R}^n$. Let $g:\mathbb{R}\to \mathbb{R}^n$ be any surjective function. Define $f_\theta :\mathbb{R}^n\to \mathbb{R}$ by
$$f_\theta (x_1,\cdots, x_n):=\frac{1}{(2\pi )^{n/2}}\exp \left\{-\frac{1}{2}\left(\sum_{i=1}^nx_i^2\right)+\theta \left(\sum_{i=1}^nx_i\right)-\frac{n}{2}\theta ^2\right\}\mathbf{1}_{\mathbb{R}^n\setminus \{g(\theta )\}}(x_1,\cdots, x_n),$$
where $\mathbf{1}_{\mathbb{R}^n\setminus \{g(\theta )\}}$ denotes the indicator function of the set $\mathbb{R}^n\setminus \{g(\theta )\}$. Since $\lambda^n(\{g(\theta )\})=0$, $f_\theta $ is a joint density of $X:=(X_1,\cdots, X_n)$. Let $T:\mathbb{R}^n\to \mathbb{R}^n$ be the identity function. Since $g$ is surjective, we infer that, for any $x,y\in\mathbb{R}^n$, we have $T(x)=T(y)$ iff there exists a finite constant $h_{xy}>0$ independent of $\theta$ such that $f_\theta (y)=f_\theta (x)h_{xy}$ for every $\theta\in\mathbb{R}$. (Indeed, if $x\neq y$, we can choose a $\theta\in\mathbb{R}$ such that $g(\theta )=y$ due to surjectivity. For this $\theta$, the indicator function in the definition of $f_\theta$ causes $f_\theta (y)=0$, while $f_\theta (x)\neq 0$, making the equality $f_\theta (y)=f_\theta (x)h_{xy}$ impossible for any $h_{xy}>0$.) Hence, according to Criterion \ref{crit:2.1}, $X=T(X)$ is a minimal sufficient statistic, which implies, recalling that $\sum _{i=1}^nX_i$ is sufficient, that there exists a measurable function $u:\mathbb{R}\to \mathbb{R}^n$ such that $(x_1,\cdots, x_n)=u(\sum _{i=1}^nx_i)$ $P_\theta$-a.e. for every $\theta\in\mathbb{R}$, where $P_\theta $ is the probability measure given by $P_\theta (E):=\int _Ef_\theta d\lambda ^n$. Since $f_\theta >0$ $\lambda ^n$-a.e. for every $\theta\in\mathbb{R}$, we conclude that $(x_1,\cdots, x_n)=u(\sum _{i=1}^nx_i)$ $\lambda ^n$-a.e., a contradiction. To see this, define $A:=\{x\in\mathbb{R}^n:u(\sum _{i=1}^nx_i)=x\}$ and $V:=\{x\in\mathbb{R}^n:\sum _{i=1}^nx_i=0\}$. It is straightforward to prove that $A$ is measurable and $(A-A)\cap (V\setminus \{0\})=\emptyset$, in which $A-A:=\{a_1-a_2:a_1,a_2\in A\}$. However, if $\lambda ^n(A)>0$, then, by Steinhaus Theorem, we know that $A-A$ contains an open ball centered at $0$, allowing us to conclude that $(A-A)\cap (V\setminus \{0\})\neq \emptyset $.
\end{counterexample}

Observe that, for each $\theta$, the function $f_\theta$ differs from the usual $N(\theta,1)^{\otimes n}$ joint density only on the $\lambda^n$-null set $\{g(\theta)\}$. Hence the induced measure $P_\theta(E):=\int_E f_\theta\, d\lambda^n$ coincides with the usual $N(\theta,1)^{\otimes n}$ law. The counterexample therefore exploits the non-uniqueness of Radon--Nikodym derivatives (choice of version), not a change in the statistical model.

The previous counterexample demonstrates that Criterion \ref{crit:2.1}, as commonly stated in the literature, does not hold in general. The proof that often accompanies this criterion, while seemingly correct, typically contains a subtle error. The flaw lies in an imprecise application of the Neyman-Fisher Factorization Theorem. Specifically, this factorization theorem states that if $f_\theta $ is a joint density of a random sample $X$ and $T(X)$ is a sufficient statistic, then there exist measurable functions $g_\theta$ and $h$ such that $f_\theta =g_\theta (T)h$ almost everywhere. The theorem, however, does not guarantee that this equality holds for every point. Therefore, for a pre-specified density $f_\theta$, a specific sample point $x$, and sufficient statistic $T$, one cannot guarantee that there exist measurable functions $g_\theta$ and $h$ such that $f_\theta (x)=g_\theta (T(x))h(x)$.

We conclude this discussion with the following counterexample, which shows that Criterion \ref{crit:2.6} does not hold without additional conditions.

\begin{counterexample}
\label{ex:3.2}
Define $\mathcal{X}:=\{1,\,2,\,3,\,4\}$ and let $(\mathcal{X},\,2^\mathcal{X},\{P_\theta\}_{\theta\in (0,1)})$ be a statistical model, where $2^\mathcal{X}$ is the collection of all subsets of $\mathcal{X}$ and, for every $\theta\in(0,1)$, $P_\theta$ is the probability on $2^\mathcal{X}$ given by $P_\theta (\{1\})=\theta /3$, $P_\theta (\{2\})=(2\theta )/3$, $P_\theta (\{3\})=(1-\theta )/3$, and $P_\theta (\{4\})=2(1-\theta )/3$. For each $\theta\in(0,1)$, define $p_\theta:\mathcal{X}\to \mathbb{R}$ by $p_\theta (i):=P_\theta (\{i\})$. Note that, for each $\theta\in(0,1)$, $p_\theta$ is a density of $P_\theta$ w.r.t. the counting measure on $\mathcal{X}$ (note that, in this case, the counting measure is $\sigma$-finite, since $\mathcal{X}$ is a finite set). The space is a standard Borel space \citep[see the first example of Section 2.2 of][]{Srivastava1998}. Define $T:\mathcal{X}\to \{0,\,1\}$ by $T:=\mathbf{1}_{\{1,2\}}$. Then $T$ is measurable w.r.t. $\sigma$-algebras $2^\mathcal{X}$ and $2^{\{0,1\}}$. Next, we will prove that $T$ is sufficient.

We have $p_\theta (x)=g_{\theta}(T(x))h(x)$ for any $\theta\in(0,1)$ and $x\in \mathcal{X}$, where $h:\mathcal{X}\to \mathbb{R}$ is given by $h(1)=h(3)=1/3$ and $h(2)=h(4)=2/3$ and, for each $\theta\in(0,1)$, $g_\theta :\{0,1\}\to \mathbb{R}$ is given by $g_\theta (0):=1-\theta $ and $g_\theta (1):=\theta $. Note that $h$ and $g_\theta$ are non-negative measurable functions for every $\theta\in (0,1)$. Hence, using the Neyman-Fisher Factorization Theorem, we conclude that $T$ is sufficient.

Define $U:\mathcal{X}\to \mathcal{X}$ by $U(x):=x$. Then $U$ is measurable. We will now prove that $U$ is minimal sufficient using Pfanzagl's method. We can write $p_\theta (x)=\tilde g_{\theta}(U(x))\tilde h(x)$ for any $\theta\in (0,1)$ and $x\in \mathcal{X}$, where $\tilde h:\mathcal{X}\to \mathbb{R}$ is given by $\tilde h(x):=1$, and $\tilde g_\theta :\mathcal{X}\to \mathbb{R}$ (for each $\theta\in (0,1)$) is given by $\tilde g_\theta := p_\theta $. Note that $\tilde h$ and $\tilde g_\theta$ are non-negative measurable functions for every $\theta\in(0,1)$. It is straightforward to verify that, for any $t_1,t_2\in \mathcal{X}$ satisfying $\tilde g_\theta (t_1)=\tilde g_\theta (t_2)$ for every $\theta\in (0,1)\cap \mathbb{Q}$, we have $t_1=t_2$. Therefore, as the set $(0,1)\cap \mathbb{Q}$ is countable, we conclude, using Pfanzagl's method, that $U$ is minimal sufficient. Hence, since $T$ is sufficient, we conclude that there exists a measurable function $f:\{0,1\}\to \mathcal{X}$ such that for every $\theta\in (0,1)$, we have $U=f(T)$ $P_\theta$-a.e., implying that $U=f(T)$, because $P_\theta (\{i\})>0$ for any $\theta\in (0,1)$ and $i\in \mathcal{X}$.

Therefore, $1=U(1)=f(\mathbf{1}_{\{1,2\}}(1))=f(\mathbf{1}_{\{1,2\}}(2))=U(2)=2$, a contradiction.
\end{counterexample}

The preceding counterexample shows that Theorem 1.4.4 in \cite{Pfanzagl1994}, which corresponds to Criterion \ref{crit:2.6}, does not hold as stated. The gap in its proof lies in the application of an unjustified argument from Theorem 1.4.2 of that work. We now outline Pfanzagl's argument to identify where the error occurs, adopting our notation instead of that of the original text.

Let $(\mathcal{X},\Sigma,\{P_\theta\}_{\theta\in\Theta })$ be a statistical model and $\mu :\Sigma\to\overline{\mathbb{R}}$ a $\sigma$-finite measure such that $P_\theta \ll \mu $ for every $\theta\in\Theta$. Suppose that $(\mathcal{X},\Sigma )$ is countably generated. Let $(\mathcal{T},\Sigma _\mathcal{T})$ be a standard Borel space and $T:\mathcal{X}\to \mathcal{T}$ a measurable function. For each $\theta\in\Theta$, let $f_\theta^\mu :\mathcal{X}\to \mathbb{R}$ be a density of $P_\theta$ w.r.t. $\mu$. We assume the following conditions: (a) for each $\theta\in\Theta$, we have $f_\theta^\mu =g_\theta^\mu (T)h^\mu$, where $h^\mu:\mathcal{X}\to \mathbb{R}$ and $g_\theta^\mu :\mathcal{T}\to \mathbb{R}$ are two non-negative measurable functions; and (b) there exists a subset $\{\theta_n\}_{n\in\mathbb{N}}\subseteq\Theta$ such that, for any $t_1,t_2\in \mathcal{T}$ satisfying $g_{\theta_n}^\mu (t_1)=g_{\theta_n}^\mu (t_2)$ for every $n\in\mathbb{N}$, we have $t_1=t_2$.

Define $H^\mu :\mathcal{T}\to \mathbb{R}^\mathbb{N}$ by $H^\mu (t):=(g_{\theta _n}^\mu(t))_{n\in\mathbb{N}}$. In the proof of Theorem 1.4.4, Pfanzagl asserts that the statistic $T^{\mu }:\mathcal{X}\to \mathbb{R}^\mathbb{N}$ given by $T^\mu (x):=H^\mu (T(x))$ is minimal sufficient, invoking the proof of Theorem 1.4.2. However, the proof of Theorem 1.4.2 establishes only the existence of a probability measure $\mathbb{P}:\Sigma \to \mathbb{R}$ such that the statistic $T^\mathbb{P}:\mathcal{X}\to \mathbb{R}^\mathbb{N}$ given by $T^\mathbb{P}(x):=H^\mathbb{P}(T(x))$ is minimal sufficient, where $H^\mathbb{P}:\mathcal{T}\to \mathbb{R}^\mathbb{N}$ is defined by $H^\mathbb{P}(t):=(g_{\theta _n}^\mathbb{P}(t))_{n\in\mathbb{N}}$ with components $g_{\theta _n}^\mathbb{P}$ chosen such that the composition $g_{\theta _n}^\mathbb{P}(T)$ is a density of $P_{\theta _n}$ w.r.t. $\mathbb{P}$. Thus, the argument in Theorem 1.4.2 is purely existential; it constructs a specific collection of measurable functions yielding a minimal sufficient statistic, and does not imply that an arbitrary pre-specified collection, such as $\{ g_{\theta _n}^\mu \}_{n\in\mathbb{N}}$, will produce a minimal sufficient statistic of the form $H(T(x))$. Moreover, Counterexample \ref{ex:3.2} demonstrates that this gap cannot be closed without further assumptions.

We conclude this section with a simple observation. Since the counting measure on a non-empty finite set, normalized by its cardinality, is a probability measure, we may multiply the densities in Counterexample~\ref{ex:3.2} by $4$; this shows that Criterion~\ref{crit:2.6} remains false even if we additionally require $\mu$ to be a probability measure equivalent to $\{P_\theta\}_{\theta\in\Theta}$ (i.e., $\mu(E)=0$ if and only if $P_\theta(E)=0$ for every $\theta\in\Theta$).

\section{Corrected and Generalized Methods}
\label{sec:methods}

This section introduces version-robust and correct criteria for identifying minimal sufficient statistics, together with examples illustrating their application. The main idea is to replace pointwise likelihood comparisons over the whole parameter space by arguments based on countable subfamilies or verifiable approximation hypotheses. These criteria are rigorously proved in Section \ref{sec:proofs}.

These methods are particularly useful in settings in which pointwise likelihood comparisons are delicate but still informative, such as symmetry models, models with parameter-dependent support, and models on analytic Borel sample spaces. In particular, Examples \ref{ex:4.1}--\ref{ex:4.4} illustrate Method \ref{meth:3.1}, Example \ref{ex:4.5} illustrates Method \ref{meth:3.2}, and Example \ref{ex:4.6} illustrates Method \ref{meth:3.3}.

\begin{method}
\label{meth:3.1}
Let $(\mathcal{X},\Sigma ,\{P_\theta\}_{\theta\in\Theta })$ be a statistical model and $\mu:\Sigma\to\overline{\mathbb{R}}$ a $\sigma$-finite measure such that $P_\theta \ll \mu$ for every $\theta\in\Theta$. Suppose that $(\mathcal{X},\Sigma )$ is an analytic Borel space. For each $\theta\in\Theta$, let $f_\theta :\mathcal{X}\to \mathbb{R}$ be a density of $P_\theta $ w.r.t. $\mu$. Let $(\mathcal{T},\Sigma _\mathcal{T})$ be a separable measurable space and $T:\mathcal{X}\to \mathcal{T}$ a measurable function. For any $x\in \mathcal{X}$ and $\Theta'\subseteq\Theta$, define
$$D(x,\Theta '):=\big\{y\in \mathcal{X}:(\exists h_{xy}\in (0,\infty ))(\forall \theta \in \Theta ')\big(f_\theta (y)=f_\theta (x)h_{xy}\big)\big\}.$$
If $T$ is sufficient and there exists a non-empty countable subset $\Theta _0\subseteq\Theta$ such that, for any $x,y\in\mathcal{X}$ satisfying $y\in D(x,\Theta _0)$, we have $T(x)=T(y)$, then $T$ is minimal sufficient.
\end{method}

Counterexample \ref{ex:3.1} shows that if one permits $\Theta_0$ to be uncountable, then pointwise statements involving the densities $\{dP_\theta/d\mu\}_{\theta\in\Theta_0}$ become vulnerable to version dependence: one can choose versions in a $\theta$-dependent way on $\mu$-null sets and thereby alter the induced proportionality relation. In contrast, Method \ref{meth:3.1} avoids this pitfall by restricting to a \emph{countable} subfamily $\Theta_0$, which enables the selection of versions that are consistent simultaneously for all $\theta\in\Theta_0$ outside a single $\mu$-null set.

The following examples illustrate the application of Method \ref{meth:3.1}. The minimal sufficient statistic of Example \ref{ex:4.1} appears in \citet{Thomas2021}.

\begin{example}
\label{ex:4.1}
Let $\Theta$ be the set of all probability densities $f:\mathbb{R}\to \mathbb{R}$ (w.r.t. the Lebesgue measure) such that $f(x)=f(-x)$ for every $x\in\mathbb{R}$. Let $X_1,\cdots, X_n$ be a random sample with density $f\in \Theta $. Then $p_f:\mathbb{R}^n\to \mathbb{R}$ given by $p_f(x_1,\cdots, x_n):=\prod _{i=1}^nf(x_i)$ is a joint density of $(X_1,\cdots,X_n)$. Since $f$ is symmetric at $0$ for every $f\in \Theta$, we have $p_f(x_1,\cdots, x_n)=\prod _{i=1}^nf(|x_i|)$ for every $f\in \Theta$, allowing us to conclude, using the Neyman-Fisher Factorization Theorem, that the statistic $(|X|_{(1)},\cdots,|X|_{(n)})$ is sufficient, where $|X|_{(i)}$ denotes the $i$-th order statistic of $(|X_1|,\cdots, |X_n|)$. Next, we will prove that $(|X|_{(1)},\cdots,|X|_{(n)})$ is also minimal.

For each $\alpha >0$, define $f_\alpha:\mathbb{R}\to \mathbb{R}$ by $f_\alpha (x):=\frac{1}{\pi\alpha (1+(x/\alpha )^2)}$. That is, $f_\alpha $ is a density function of a random variable with distribution $\text{Cauchy}(0,\alpha )$. Hence, $f_\alpha \in \Theta$ for every $\alpha >0$. Define $\Theta _0:=\{f_\alpha :\alpha \in \mathbb{Q}_+\}$, where $\mathbb{Q}_+:=\mathbb{Q}\cap (0,\infty )$. Then $\Theta _0$ is a countable subset of $\Theta$. For every $x\in\mathbb{R}^n$, define $D(x ,\Theta _0):=\big\{y\in \mathbb{R}^n:(\exists h_{xy}\in (0,\infty ))(\forall f\in \Theta_0 )\big(p_f(y)=p_f (x)h_{xy}\big)\big\}$.

Let $x:=(x_1,\cdots, x_n),y:=(y_1,\cdots, y_n)\in\mathbb{R}^n$ be arbitrary such that $y\in D(x,\Theta _0)\,\,(i)$. Because of Method \ref{meth:3.1}, to prove that $(|X|_{(1)},\cdots,|X|_{(n)})$ is minimal, it suffices to prove that $(|x|_{(1)},\cdots, |x|_{(n)})=(|y|_{(1)},\cdots, |y|_{(n)})$.

From $(i)$, we know that there exists $h_{xy}>0$ such that $p_{f_\alpha }(y)=p_{f_\alpha }(x)h_{xy}$ for every $\alpha \in\mathbb{Q}_+$. Hence, for each $\alpha \in \mathbb{Q}_+$, we have $\prod _{k=1}^n\frac{1}{\pi\alpha (1+(y_k/\alpha )^2)}=h_{xy}\prod _{k=1}^n\frac{1}{\pi\alpha (1+(x_k/\alpha )^2)}$, implying that $P(\alpha )=Q(\alpha )\,\,(ii)$ for every $\alpha \in \mathbb{Q}_+$, where $P,Q:\mathbb{R}\to \mathbb{R}$ are the polynomials given by $P(\alpha ):=\prod _{k=1}^n(\alpha ^2+x_k^2)$ and $Q(\alpha ):=h_{xy}\prod _{k=1}^n(\alpha ^2+y_k^2)$. Since $P,Q$ are even and continuous and $\mathbb{Q}_+$ is dense in $[0,\infty )$, we conclude, using $(ii)$, that $P=Q$, allowing us to infer that $P$ and $Q$ have the same roots. Hence, $(|x|_{(1)},\cdots, |x|_{(n)})=(|y|_{(1)},\cdots, |y|_{(n)})$, since $\{\pm ix_k\}_{k=1}^n$ and $\{\pm iy_k\}_{k=1}^n$ are the sets that contain all roots of $P$ and $Q$, respectively, with $i$ denoting the complex number that satisfies $i^2=-1$.
\end{example}

\begin{example}
\label{ex:4.2}
Let $f:\mathbb{R}\to (0,\infty )$ be an integrable function. For each $\theta\in\mathbb{R}$, define $c(\theta ):=1/\int_\theta ^\infty f(x)dx$. Let $X_1,\cdots,X_n$ be a random sample with density $p_\theta (x)=c(\theta )f(x)$, for $x>\theta $ and $p_\theta (x)=0$, otherwise. We will prove that the order statistic $X_{(1)}=\min (X_1,\cdots, X_n)$ is minimal sufficient. The function $f_\theta :\mathbb{R}^n\to \mathbb{R}$ given by
$$f_{\theta }(x_1,\cdots,x_n):=(c(\theta ))^nf(x_1)\cdots f(x_n)\mathbf{1}_{(\theta ,\infty )}(x_{(1)})\,\,(i)$$
is a joint density of $(X_1,\cdots, X_n)$ w.r.t. the Lebesgue measure on $\mathbb{R}^n$. Using the Neyman-Fisher Factorization Theorem, we infer that $X_{(1)}$ is sufficient. Define $\Theta _0:=\mathbb{Q}$. Then $\Theta _0\subseteq\mathbb{R}$ is countable. For each $x\in\mathbb{R}^n$, define
$$D(x,\Theta _0):=\big\{y\in\mathbb{R}^n:(\exists h_{xy}\in (0,\infty ))(\forall \theta \in \Theta _0)\big(f_\theta (y)=f_\theta (x)h_{xy}\big)\big\}.$$
Let $x,y\in\mathbb{R}^n$ be arbitrary elements such that $y\in D(x,\Theta _0)\,\,(ii)$. Because of Method \ref{meth:3.1}, to prove that $X_{(1)}$ is minimal sufficient, it suffices to prove that $x_{(1)}=y_{(1)}$. From $(ii)$, we know that there exists a finite constant $h_{xy}>0$ such that $f_\theta (y)=f_\theta (x)h_{xy}\,\,(iii)$ for every $\theta\in \Theta _0$. Suppose, by contradiction, that $x_{(1)}\neq y_{(1)}$. Then $x_{(1)}<y_{(1)}$ or $y_{(1)}<x_{(1)}$. Suppose that $x_{(1)}<y_{(1)}$. Then there exists a rational number $\theta _0\in \Theta _0:=\mathbb{Q}$ such that $x_{(1)}<\theta _0<y_{(1)}\,\,(iv)$. From $(iii)$, we know that $f_{\theta_0} (y)=f_{\theta_0} (x)h_{xy}$, which implies, using $(i)$ and $(iv)$, that $f_{\theta _0}(y)=f_{\theta _0}(x)h_{xy}=0$, since $f_{\theta _0}(x)=0$. However, this is a contradiction since $f_{\theta _0}(y)>0$. Therefore, the assumption that $x_{(1)}<y_{(1)}$ is false. Analogously, we can conclude that $y_{(1)}<x_{(1)}$ does not hold, implying that $x_{(1)}=y_{(1)}$.
\end{example}

\begin{example}
\label{ex:4.3}
Let $X_1,\cdots, X_n$ be a random sample with density $p_\theta $ given by
$$p_\theta (x)=\left(\frac{2}{\pi }\right)^{1/2}\exp \left\{-\frac{(x-\theta )^2}{2}\right\}\mathbf{1}_{[\theta ,\infty )}(x),\,\,\,\,\theta \in \mathbb{R}.$$
We will prove that $(\overline X,X_{(1)})$ is a minimal sufficient statistic. The function $f_\theta :\mathbb{R}^n\to \mathbb{R}$ given by $f_\theta (x_1,\cdots, x_n)=(2/\pi )^{n/2}e^{n\theta \overline x-n\theta ^2/2}e^{-(\sum _{i=1}^nx_i^2)/2}\mathbf{1}_{[\theta ,\infty )}(x_{(1)})\,\,(i)$ is a joint density of $(X_1,\cdots, X_n)$ w.r.t. the Lebesgue measure on $\mathbb{R}^n$. Using the Neyman-Fisher Factorization Theorem, we conclude that $(\overline X,X_{(1)})$ is a sufficient statistic. Define $\Theta _0:=\mathbb{Q}$. Then $\Theta _0\subseteq\mathbb{R}$ is countable. For each $x\in\mathbb{R}^n$, define
$$D(x,\Theta _0):=\big\{y\in\mathbb{R}^n:(\exists h_{xy}\in (0,\infty ))(\forall \theta \in \Theta _0)\big(f_\theta (y)=f_\theta (x)h_{xy}\big)\big\}.$$
Let $x,y\in\mathbb{R}^n$ be arbitrary elements such that $y\in D(x,\Theta _0)\,\,(ii)$. Because of Method \ref{meth:3.1}, to prove that $(\overline X,X_{(1)})$ is minimal sufficient, it suffices to prove that $(\overline x,x_{(1)})=(\overline y,y_{(1)})$. From $(ii)$, we know that there exists a finite constant $h_{xy}>0$ such that $f_\theta (y)=f_\theta (x)h_{xy}$ for every $\theta\in\Theta _0$. Therefore, using $(i)$, we conclude that for every $\theta\in\Theta _0$, we have
$$e^{n\theta \overline y }e^{-(\sum _{i=1}^ny_i^2)/2}\mathbf{1}_{[\theta ,\infty )}(y_{(1)})=e^{n\theta \overline x }e^{-(\sum _{i=1}^nx_i^2)/2}\mathbf{1}_{[\theta ,\infty )}(x_{(1)})h_{xy}\,\,(iii).$$
Repeating an argument from Example \ref{ex:4.2}, we can conclude that $x_{(1)}=y_{(1)}$, which implies, using $(iii)$, that, for each $\theta\in\Theta _0\cap (-\infty ,x_{(1)}]$, we have $e^{n\theta (\overline y-\overline x) }=e^{-(\sum _{i=1}^nx_i^2)/2}h_{xy}e^{(\sum _{i=1}^ny_i^2)/2}$. Since the previous equality holds for every $\theta\in\Theta _0\cap (-\infty ,x_{(1)}]$ and its right side does not depend on $\theta$, we can conclude that $\overline y-\overline x=0$ and, consequently, that $\overline x=\overline y$. Hence, $(\overline x,x_{(1)})=(\overline y,y_{(1)})$.
\end{example}

\begin{example}
\label{ex:4.4}
Let $(X_1,Y_1),\cdots, (X_n,Y_n)$ be a random sample with density $p_\theta (x,y)=2/\theta ^2$, for $x>0$, $y>0$, and $x+y<\theta$, and $p_\theta (x,y)=0$, otherwise, with $\theta >0$. Define $X:=(X_1,\cdots, X_n)$ and $Y:=(Y_1,\cdots, Y_n)$. We will prove that the order statistic $(X+Y)_{(n)}$ is minimal sufficient. The function $f_\theta :((0,\infty )^2)^n\to \mathbb{R}$ given by
$$f_\theta ((x_1,y_1),\cdots, (x_n,y_n)):=2^n\theta^{-2n}\mathbf{1}_{(0,\infty )^2}(x_{(1)},y_{(1)})\mathbf{1}_{(-\infty ,\theta )}((x+y)_{(n)})\,\,(i)$$
is a joint density of $((X_1,Y_1),\cdots, (X_n,Y_n))$ w.r.t. the Lebesgue measure on $((0,\infty )^2)^n$. Using Propositions 8.1.2, 8.1.3, and 8.1.7 of \cite{Cohn2013}, we can conclude that the product $((0,\infty )^2)^n$ of $n$ copies of $(0,\infty )^2$ together with its Borel $\sigma$-algebra is a standard Borel space. Using $(i)$ and the Neyman-Fisher Factorization Theorem, we can conclude that the statistic $(X+Y)_{(n)}$ is sufficient. Define $\Theta _0:=\mathbb{Q}\cap (0,\infty )$. Then $\Theta _0\subseteq (0,\infty )$ is countable. For each $a\in ((0,\infty )^2)^n$, define $D(a,\Theta _0):=\big\{b\in((0,\infty )^2)^n:(\exists h_{ab}\in (0,\infty ))(\forall \theta \in \Theta _0)\big(f_\theta (b)=f_\theta (a)h_{ab}\big)\big\}$.
Let $a:=((x_1,y_1),\cdots,(x_n,y_n)),b:=((z_1,w_1),\cdots,(z_n,w_n))\in ((0,\infty )^2)^n$ be arbitrary such that $b\in D(a,\Theta _0)\,\,(ii)$. Because of Method \ref{meth:3.1}, to prove that $(X+Y)_{(n)}$ is minimal sufficient, it suffices to prove that $(x+y)_{(n)}=(z+w)_{(n)}$, where $x:=(x_1,\cdots, x_n)$ and $y,z,w$ are defined likewise. From $(ii)$, we know that there exists a finite constant $h_{ab}>0$ such that $f_\theta (b)=f_\theta (a)h_{ab}$ for every $\theta\in\Theta _0$. Then, using $(i)$, we can conclude that, for each $\theta\in\Theta _0$, we have
$$\mathbf{1}_{(-\infty ,\theta )}((z+w)_{(n)})=\mathbf{1}_{(-\infty ,\theta )}((x+y)_{(n)})h_{ab}\,\,(iii).$$
Suppose, by contradiction, that $(x+y)_{(n)}<(w+z)_{(n)}$. Then, using $(iii)$, we obtain $0=h_{ab}$ for every $\theta \in ((x+y)_{(n)},(z+w)_{(n)})\cap \Theta _0$, a contradiction. Hence, $(w+z)_{(n)}\leq (x+y)_{(n)}$. Analogously, we can conclude that $(x+y)_{(n)}\leq (w+z)_{(n)}$, implying that $(x+y)_{(n)}=(z+w)_{(n)}$.
\end{example}

Recall that, in Method \ref{meth:3.1}, the implication $y \in D(x,\Theta_0)\ \Longrightarrow\ T(x)=T(y)$ must hold for all $x,y \in \mathcal{X}$. This makes the method not directly applicable to some statistics, as the following example shows.
\begin{example} Let $(X_1,X_2)$ be a vector with density $f_\theta (x_1,x_2)= (4/\pi) \theta^3 x_1^2 x_2^2 \exp(-\theta(x_1^2 + x_2^2))$, for $x_1, x_2\in \mathbb{R}$, and $\theta>0$. By the Neyman-Fisher Factorization Theorem, $T(X_1,X_2) = X_1^2+X_2^2$ is sufficient. Note that $f_\theta(x_1,x_2)= 0$ whenever $x_1 x_2=0$. Take, for instance, $x=(0,2)$ and $y=(0,1)$. Then 
$T(x)=4 \neq T(y)=1$, but, as $f_\theta(x)=f_\theta(y)=0$ for every $\theta>0$, we have, for any countable set $\Theta_0\subset (0,\infty)$, that $y \in D(x,\Theta_0)$. Therefore, Method \ref{meth:3.1} cannot be applied directly on $T$, since the required implication  ``$y \in D(x, \Theta_0)$ implies $T(x) = T(y)$'' fails. However, we can define a pointwise modified statistic which is equal to $T$ $P_\theta$-a.e. for each $\theta>0$. For example, let $\tilde T(x_1,x_2) = T(x_1, x_2)$, if $x_1x_2 \neq 0$, and $\tilde T(x_1,x_2) = 1$, otherwise. Now let $\Theta _0:=\mathbb{Q}\cap (0,\infty )$ and assume that $y \in D(x, \Theta_0)$. Then there exists
$h_{xy} \in (0, \infty)$ such that
$f_\theta(y)=f_\theta(x)\,h_{xy}$, $\forall \theta\in\Theta_0$.

There are two cases, namely, (1) if $x_1x_2=0$, then $f_\theta(x)=0$ for every $\theta\in\Theta_0$. Hence $f_\theta(y)=0$ for every $\theta\in\Theta_0$ as well, so necessarily $y_1y_2=0$. Therefore, $\tilde T(x)=1=\tilde T(y)$; and (2) if $x_1x_2\neq0$, then $f_\theta(x)>0$ for every $\theta\in\Theta_0$. Since $y \in D(x,\Theta_0)$ and $h_{xy}>0$, we must also have $y_1y_2\neq0$. In this case,
\[
h_{xy} = \frac{f_\theta(y)}{f_\theta(x)}
= \frac{y_1^2 y_2^2}{x_1^2 x_2^2}
\exp\big(-\theta(\tilde T(y)-\tilde T(x))\big),
\qquad \forall \theta\in\Theta_0.
\]
Because $h_{xy}$ does not depend on $\theta$ and $\Theta_0$ contains infinitely many distinct values, we must have $\tilde T(x) = \tilde T(y)$. Thus the implication required in Method \ref{meth:3.1} holds for $\tilde T$. Since
$\tilde T=T$ $P_\theta$-a.e. for every $\theta>0$, $\tilde T$ is also sufficient.
Hence, by Method \ref{meth:3.1}, $\tilde T$ is minimal sufficient. Moreover, since
$T=\tilde T$ $P_\theta$-almost surely for every $\theta>0$, $T$ is also minimal sufficient.
\end{example}

The following method is a generalization of the one introduced by \cite{Sato1996}, whose original proof was restricted to Euclidean spaces. Its formal proof is deferred to Section \ref{sec:proofs}.

\begin{method}[Sato]
\label{meth:3.2}
Let $(\mathcal{X},\Sigma ,\{P_\theta\}_{\theta\in\Theta })$ be a statistical model and $\mu:\Sigma\to\overline{\mathbb{R}}$ a $\sigma$-finite measure such that $P_\theta \ll \mu$ for every $\theta\in\Theta$. Suppose that $(\mathcal{X},\Sigma )$ is an analytic Borel space, $(\mathcal{T},\Sigma_\mathcal{T})$ is a separable measurable space, and $T:\mathcal{X}\to\mathcal{T}$ is a measurable function. For each $\theta\in\Theta$, let $f_\theta :\mathcal{X}\to \mathbb{R}$ be a density of $P_\theta $ w.r.t. $\mu$. Suppose that there exists a non-empty countable subset $\Theta _0\subseteq\Theta$ such that, for each $\theta\in\Theta$, there exists a sequence $(\theta _n)_{n\in\mathbb{N}}$ of $\Theta _0$ such that the limit $\lim_{n\to\infty}f_{\theta_n}$ exists in $\mathbb{R}$ $\mu$-a.e. and for every $x\in\mathcal{X}$ we have
$$f_\theta (x)=\begin{cases}\lim_{n\to \infty}f_{\theta _n}(x),&\text{if the limit exists}\\0,&\text{otherwise}\end{cases}.$$
Define $D(x):=\big\{y\in \mathcal{X}:(\exists h_{xy}\in (0,\infty ))(\forall \theta \in \Theta )\big(f_\theta (y)=f_\theta (x)h_{xy}\big)\big\}$. If for any $x,y\in\mathcal{X}$, we have $T(x)=T(y)$ iff $y\in D(x)$, then $T$ is minimal sufficient.
\end{method}

This method states that, under the approximation condition above, the usual likelihood-ratio characterization of minimal sufficiency becomes valid: $T$ is minimal sufficient if, for any sample points $x,y\in\mathcal{X}$, we have $T(x)=T(y)$ iff there exists a finite constant $h_{xy}>0$ independent of $\theta $ such that $f_\theta (y)=f_\theta (x)h_{xy}$ for every $\theta\in\Theta$.

Note that the latter approach highlights the necessary conditions under which Criterion \ref{crit:2.1} introduced in Section~\ref{sec:intro}  becomes valid. It is particularly useful when $\Theta$ is a subset of $\mathbb{R}^n$ and the densities $f_\theta$ are continuous w.r.t. $\theta$, but it is difficult to apply in cases such as Example~\ref{ex:4.1}. The following example illustrates its application.

\begin{example}
\label{ex:4.5}
Let $X_1,\cdots, X_n\sim \text{Cauchy}(\theta ,1),\theta\in\mathbb{R}$ be a random sample. The function $f_\theta :\mathbb{R}^n\to \mathbb{R}$ given by $f_\theta (x_1,\cdots, x_n):=\pi ^{-n}\prod_{i=1}^n\frac{1}{1+(x_i-\theta )^2}$ is a joint density of $(X_1,\cdots, X_n)$ w.r.t. the Lebesgue measure on $\mathbb{R}^n$. We will show that $(X_{(1)},X_{(2)},\cdots, X_{(n)})$ is a minimal sufficient statistic, where $X_{(k)}$ denotes the $k$-th order statistic. Since the function $\mathbb{R}\to \mathbb{R},\,\theta \mapsto f_\theta (x)$ is continuous for every $x\in\mathbb{R}^n$ and the set of rational numbers $\mathbb{Q}$ is countable and dense in $\mathbb{R}$, the conditions of Method \ref{meth:3.2} are satisfied. Define
$$D(x ):=\big\{y\in \mathbb{R}^n:(\exists h_{xy}\in (0,\infty ))(\forall \theta \in \mathbb{R})\big(f_\theta (y)=f_\theta (x)h_{xy}\big)\big\}$$
for every $x\in\mathbb{R}^n$. Since $h/f_\theta (x)$ is a polynomial in $\theta$ for every $h>0$ and $x:=(x_1,\cdots,x_n)\in\mathbb{R}^n$ with complex roots $x_j\pm i,j\in \{1,\cdots,n\}$, as two identical polynomials have the same roots, it is straightforward to show that, for any two sample points $x,y\in\mathbb{R}^n$, we have $(x_{(1)},\cdots, x_{(n)})=(y_{(1)},\cdots, y_{(n)})$ iff $y\in D(x)$. Hence, using Method \ref{meth:3.2}, we infer that $(X_{(1)},X_{(2)},\cdots, X_{(n)})$ is minimal sufficient.
\end{example}
We conclude this section with a method for exponential models, based on Proposition 1.6.9 in \cite{Pfanzagl1994}. Pfanzagl's original proof invokes his minimality criterion (cf. Criterion~\ref{crit:2.6}), which we showed to be false without extra assumptions in Section~\ref{sec:counterexamples}. We provide a complete proof of the following method in Section~\ref{sec:proofs}; our version holds under slightly stronger hypotheses than Pfanzagl's original formulation.

\begin{method}
\label{meth:3.3}
Let $(\mathcal{X},\Sigma ,\{P_\theta\}_{\theta\in\Theta })$ be a statistical model and $\mu:\Sigma\to\overline{\mathbb{R}}$ a $\sigma$-finite measure such that $P_\theta \ll \mu$ for every $\theta\in\Theta$. Suppose that $(\mathcal{X},\Sigma )$ is an analytic Borel space. For each $\theta\in\Theta$, let $f_\theta :\mathcal{X}\to \mathbb{R}$ be a density of $P_\theta $ w.r.t. $\mu$. Suppose that there exist functions $\eta _1,\cdots,\eta _k,B:\Theta \to \mathbb{R}$ and measurable functions $T_1,\cdots ,T_k,h:\mathcal{X}\to \mathbb{R}$ such that, for each $\theta\in\Theta$, we have
$$f_\theta (x)=\exp \big[\big(\sum _{i=1}^k\eta_i (\theta )T_i(x)\big)-B(\theta )\big]h(x)\quad \mu\text{-a.e.}$$
If for any $a_0,\cdots,a_k\in\mathbb{R}$ satisfying $\sum _{i=1}^ka_i\eta_i (\theta )=a_0$ for every $\theta\in\Theta$, we have $a_i=0$ for every $i\in \{0,\cdots, k\}$, then $T:\mathcal{X}\to \mathbb{R}^k$ given by $T(x):=(T_1(x),\cdots, T_k(x))$ is a minimal sufficient statistic.
\end{method}

The following examples illustrate the application of this method.

\begin{example}
\label{ex:4.6}
Let $X_1,\cdots ,X_n\sim N(\theta ,k\theta^2),\theta >0$ be a random sample, where $k>0$ is a known constant. We will prove that $(\sum _{i=1}^nX_i,\sum _{i=1}^nX_i^2)$ is a minimal sufficient statistic. The function $f_\theta:\mathbb{R}^n\to \mathbb{R}$ given by $$f_\theta (x_1,\cdots, x_n):=\frac{e^{-n/(2k)}}{(2\pi k)^{n/2}}\exp \left\{\frac{1}{k\theta }\left(\sum _{i=1}^nx_i\right)-\frac{1}{2k\theta ^2}\left(\sum _{i=1}^nx_i^2\right)-n\log (\theta )\right\}$$ is a joint density of $(X_1,\cdots, X_n)$ w.r.t. the Lebesgue measure on $\mathbb{R}^n$.

For each $\theta>0$, define $\eta _1(\theta ):=1/(k\theta )$ and $\eta _2(\theta ):=-1/(2k\theta ^2)$. Let $a,b,c\in\mathbb{R}$ be arbitrary elements such that $a\eta _1(\theta )+b\eta _2(\theta )=c\,\,(i)$ for every $\theta>0$. Because of Method \ref{meth:3.3}, to prove that $(\sum _{i=1}^nX_i,\sum _{i=1}^nX_i^2)$ is minimal sufficient, it suffices to prove that $a=b=c=0$. From $(i)$, we know that $2kc\theta ^2-2a\theta +b=0$ for every $\theta >0$, allowing us to conclude, noting that this is a polynomial with degree at most $2$ in $\theta$, that $2kc=-2a=b=0$ and, hence, that $a=b=c=0$.
\end{example}

\section{Proofs}
\label{sec:proofs}

Before proving the main results, it is necessary to establish some preliminaries. 

Let $(\mathcal{X},\Sigma,\{P_\theta\}_{\theta\in\Theta})$ be a statistical model. Whenever we write $f:\mathcal{X}\to(\mathcal{Y},\Sigma _\mathcal{Y})$, we mean that $(\mathcal{Y},\Sigma _\mathcal{Y})$ is a measurable space and that $f:\mathcal{X}\to\mathcal{Y}$ is $(\Sigma,\Sigma _\mathcal{Y})$-measurable. Moreover, when $f: \mathcal{X} \to \mathcal{Y}$ is $(\mathfrak{B},\Sigma _\mathcal{Y})$-measurable, where $\mathfrak{B}$ is a sub-sigma-algebra of $\Sigma$, we say that $f$ is $\mathfrak{B}$-measurable.

Let $\mathcal{X}$ be a set and $\{\mathfrak{A} _i\}_{i\in I}$ a collection of subsets of $2^\mathcal{X}$. We denote by $\vee_{i\in I}\mathfrak{A}_i$ the $\sigma$-algebra on $\mathcal{X}$ generated by the union $\cup_{i\in I}\mathfrak{A}_i$. Furthermore, we denote $\vee_{i\in I}\mathfrak{A}_i$ by $\mathfrak{A} _1\vee \cdots \vee \mathfrak{A} _n$ if $I=\{1,\cdots, n\}$.

Let $\mathcal{E}:=(\mathcal{X},\Sigma,\{P_\theta\}_{\theta\in\Theta})$ be a statistical model and $\mathfrak{A}$ a sub-$\sigma$-algebra of $\Sigma$. We denote 

$$\mathfrak{N}_\mathcal{E}:=\big\{N\in \Sigma :(\forall \theta\in\Theta )\big(P_\theta (N)=0\big)\big\}.$$

That is, $\mathfrak{N}_\mathcal{E}$ is the set of all $N\in \Sigma$ that are null with respect to every probability measure in the statistical model $\mathcal{E}$. Moreover, we denote $\overline{\mathfrak{A}}^\mathcal{E}:=\mathfrak{A}\vee \mathfrak{N}_\mathcal{E}$ for every sub-$\sigma$-algebra $\mathfrak{A}\subseteq\Sigma$. Hence, $\overline{\mathfrak{A}}^\mathcal{E}$ is the $\sigma$-algebra on $\mathcal{X}$ generated by the union $\mathfrak{A}\cup \mathfrak{N}_\mathcal{E}$. Repeating the argument in Lemma 4.1.3 of \cite{Doberkat2015}, we obtain 

$$\overline{\mathfrak{A}}^\mathcal{E}=\big\{A\vartriangle N:A\in\mathfrak{A}\wedge N\in\mathfrak{N}_\mathcal{E}\big\}.$$

Therefore, given two sub-$\sigma$-algebras $\mathfrak{A}$ and $\mathfrak{B}$ of $\Sigma$, we have $\mathfrak{A}\subseteq\overline{\mathfrak{B}}^\mathcal{E}$ iff for each $A\in\mathfrak{A}$ there exists $B\in\mathfrak{B}$ such that $P_\theta (A\vartriangle B)=0$ for every $\theta\in\Theta$.

\begin{proposition}
\label{prop:4.1}
Let $\mathcal{E}:=(\mathcal{X},\Sigma,\{P_\theta\}_{\theta\in\Theta})$ be a statistical model, $\mathfrak{A}$ a sub-$\sigma$-algebra of $\Sigma$, and $(\mathcal{Y},\Sigma_\mathcal{Y})$ a standard Borel space. If $f:\mathcal{X}\to \mathcal{Y}$ is $\overline{\mathfrak{A}}^\mathcal{E}$-measurable, then there exists a $\mathfrak{A}$-measurable function $f_0:\mathcal{X}\to\mathcal{Y}$ such that, for every $\theta\in\Theta$, we have $f=f_0$ $P_\theta$-a.e.
\end{proposition}

\begin{proof}
    See Lemma 1.10.3, p. 56, in \cite{Pfanzagl1994}
\end{proof}

The previous proposition can be generalized to separable measurable spaces, as can be seen in \citet[Proposition 2.1.4]{Cavalcante2026}.

Let $\mathcal{M}_1$, $\mathcal{M}_2$  and $\mathcal{M}$ be collections of measures on a set $(\mathcal{X},\Sigma )$. We denote $\mathcal{M}_1\ll \mathcal{M}_2$ iff for every $E\in \Sigma $ satisfying $(\forall \nu \in\mathcal{M}_2)(\nu (E)=0)$ we have $(\forall \mu \in\mathcal{M}_1)(\mu (E)=0)$. We also denote $\mathcal{M}_1\equiv \mathcal{M}_2$ iff $\mathcal{M}_1\ll \mathcal{M}_2$ and $\mathcal{M}_2\ll \mathcal{M}_1$. Moreover, we denote $\mathcal{E}\ll \mathcal{M}$ (resp. $\mathcal{E}\equiv \mathcal{M}$) iff $\{P_\theta \}_{\theta \in \Theta }\ll \mathcal{M}$ (resp. $\{P_\theta \}_{\theta \in \Theta }\equiv \mathcal{M}$), where $\mathcal{E}$ is the statistical model. If $\mathcal{M}=\{\mu\}$ is a singleton, we denote $\mathcal{E}\ll \mathcal{M}$ and $\mathcal{E}\equiv \mathcal{M}$ simply by $\mathcal{E}\ll \mu $ and $\mathcal{E}\equiv \mu $, respectively.

\begin{proposition}
\label{prop:4.2}
Let $\mathcal{E}:=(\mathcal{X},\Sigma,\{P_\theta\}_{\theta\in\Theta})$ be a statistical model and $\mu \!:\!\Sigma\to\overline{\mathbb{R}}$ a $\sigma$-finite measure. If $\mathcal{E}\ll \mu $, then there exists a countable subset $\{\theta _n\}_{n\in\mathbb{N}}\subseteq\Theta$ such that for every sequence $(c_n)_{n\in\mathbb{N}}$ of $(0,\infty )$ satisfying $\sum _{n\in\mathbb{N}}c_n=1$, we have that $P:\Sigma\to \mathbb{R}$ given by ${P}:=\sum _{n\in\mathbb{N}}c_nP_{\theta _n}$ is a probability measure such that $\mathcal{E}\equiv {P}$, $P\ll \mu $, and 

$$\frac{d{P}}{d\mu }=\sum _{n\in\mathbb{N}}c_n \frac{dP_{\theta _n}}{d\mu } \,\,\mu\text{-a.e.}$$
\end{proposition}

\begin{proof}
Suppose that $\mathcal{E}\ll \mu $. From Lemma 7 in \cite{Halmos1949}, we know that there exists a countable subset $\{\theta _n\}_{n\in\mathbb{N}}\subseteq\Theta $ such that $\mathcal{E}\equiv \{P_{\theta _n}\}_{n\in\mathbb{N}}$. Let $(c_n)_{n\in\mathbb{N}}$ be an arbitrary sequence of $(0,\infty)$ such that $\sum _{n\in\mathbb{N}}c_n=1$. Define $P:\Sigma\to \mathbb{R}$ by ${P}:=\sum _{n\in\mathbb{N}}c_nP_{\theta _n}$. Then $P$ is a probability measure. Moreover, we have ${P}\ll \{P_{\theta _n}\}_{n\in\mathbb{N}}\ll \{P_\theta \}_{\theta \in \Theta }\ll \{P_{\theta _n}\}_{n\in\mathbb{N}}\ll {P}$, implying that $\mathcal{E}\equiv {P}$.

By hypothesis, $\mathcal{E}\ll \mu $; therefore, $\{P_{\theta _n}\}_{n\in\mathbb{N}}\ll \mu $ and ${P}\ll \mu $. Hence, given the almost everywhere uniqueness of a Radon-Nikodym derivative, it is direct to verify that $\frac{d{P}}{d\mu }=\sum _{n\in\mathbb{N}}c_n \frac{dP_{\theta _n}}{d\mu }$ $\mu$-a.e.
\end{proof}

Let $\mu ,\nu $ be two measures on $(\mathcal{X},\Sigma )$ such that $\nu \ll \mu $. We denote by $[d\nu /d\mu ]$ the set of all Radon-Nikodym derivatives of $\nu$ w.r.t. $\mu$. Define $d_\mathcal{E}:\Theta\times \Theta \to \mathbb{R}$, the total variation distance, by $d_\mathcal{E}(\theta _1,\theta _2):=\sup _{E\in\Sigma }|P_{\theta _1}(E)-P_{\theta _2}(E)|.$ 

\begin{proposition}
\label{prop:4.3}
Let $\mathcal{E}:=(\mathcal{X},\Sigma,\{P_\theta\}_{\theta\in\Theta})$ be a statistical model, $\theta_1,\theta _2\in\Theta$, and $\mu:\Sigma\to\overline{\mathbb{R}}$ a $\sigma$-finite measure such that $\mathcal{E}\ll \mu $. If $f_{\theta _1}\in [dP_{\theta _1}/d\mu  ]$ and $f_{\theta _2}\in [dP_{\theta _2}/d\mu ]$, then

$$d_\mathcal{E}(\theta _1,\theta _2)=\frac{1}{2}\int _\mathcal{X}|f_{\theta _1}-f_{\theta _2}|d\mu. $$
\end{proposition}

\begin{proof}
    See Lemma 2.4 in \cite{Strasser1985}.
\end{proof}

We denote by $\tau \langle \mathcal{E}\rangle $ the coarsest topology on $\Theta$ that makes the functions $\Phi _E:\Theta \to \mathbb{R}$ given by $\Phi _E(\theta ):=P_\theta (E)$ continuous for every $E \in \Sigma$. Hence, $\tau \langle \mathcal{E}\rangle $ is the initial topology on $\Theta $ induced by the collection $\{\Phi _E:\Theta \to \mathbb{R}\}_{E\in \Sigma }$.

Let $(\mathcal{X},d)$ be a pseudo-metric space. We denote by $\tau (d)$ the topology on $\mathcal{X}$ whose elements are arbitrary unions of open balls defined by $d$. That is, $\tau (d)$ is the topology on $\mathcal{X}$ induced by the pseudo-metric $d$.

\begin{proposition}
\label{prop:4.4}
Let $\mathcal{E}:=(\mathcal{X},\Sigma,\{P_\theta\}_{\theta\in\Theta})$ be a statistical model. Then 

$$\tau \langle \mathcal{E}\rangle \subseteq \tau (d_\mathcal{E}).$$

\end{proposition}

\begin{proof}
For each $E\in\Sigma$, define $\Phi _E:\Theta\to \mathbb{R}$ by $\Phi _E(\theta):=P_\theta (E)$. To conclude the proof, it suffices to prove that $\Phi _E$ is $\tau (d_\mathcal{E})$-continuous. This follows from the observation that for any $E\in \Sigma$ and $\theta,\theta _0\in \Theta$ we have $|\Phi_E(\theta )-\Phi _E(\theta _0)|=|P_\theta (E)-P_{\theta _0}(E)|\leq d_\mathcal{E}(\theta ,\theta _0)$, by definition of the total variation distance.
\end{proof}

An immediate consequence of the previous proposition is the next corollary.

\begin{corollary}
\label{cor:4.1}
Let $\mathcal{E}:=(\mathcal{X},\Sigma,\{P_\theta\}_{\theta\in\Theta})$ be a statistical model and $\Theta _0\subseteq\Theta $. If $\Theta _0$ is dense in $(\Theta ,d_\mathcal{E})$, then $\Theta _0$ is dense in $(\Theta, \tau \langle \mathcal{E}\rangle )$.
\end{corollary}

Before stating the next proposition, we  extend the notion of (minimal) sufficiency to $\sigma$-algebras. Let $(\mathcal{X},\Sigma ,\{P_\theta \}_{\theta\in\Theta })$ be a statistical model. A sub-$\sigma$-algebra $\mathfrak{A}$ of $\Sigma$ is said to be {\bf($\mathcal{E}$-)sufficient} (or {\bf$\{P_\theta\}_{\theta\in\Theta}$-sufficient}) iff for each $E\in\Sigma$, there exists a $\mathfrak{A}$-measurable function $\kappa _E:\mathcal{X}\to \mathbb{R}$ such that for every $\theta\in\Theta$ we have $P\!_\theta (E|\mathfrak{A})=\kappa _E$ $P_\theta$-a.e. Furthermore, $\mathfrak{A}$ is said to be {\bf minimal ($\mathcal{E}$-)sufficient} (or {\bf minimal $\{P_\theta\}_{\theta\in\Theta}$-sufficient}) iff $\mathfrak{A}$ is $\mathcal{E}$-sufficient and, given any $\mathcal{E}$-sufficient sub-$\sigma$-algebra $\mathfrak{B}$ of $\Sigma$, we have $\mathfrak{A}\subseteq \overline{\mathfrak{B}}^\mathcal{E}$.

There exists a connection between sufficient statistics and sufficient $\sigma$-algebras. Before explicating this relation, we require the following lemma.

\begin{proposition}[Doob-Dynkin Lemma]
\label{doob_lemma}
Let $\mathcal{X}$ be a set, $(\mathcal{Y},\Sigma _\mathcal{Y})$ and $(\mathcal{Z},\Sigma _\mathcal{Z})$ two measurable spaces, and $f:\mathcal{X}\to \mathcal{Y}$ and $g:\mathcal{X}\to \mathcal{Z}$ two functions. If $(\mathcal{Z},\Sigma _\mathcal{Z})$ is a standard Borel space, then the following statements are \underline{equivalent}:
    \begin{enumerate}
        \item $\sigma (g)\subseteq \sigma (f)$;
        \item There exists a measurable function $h:\mathcal{Y}\to \mathcal{Z}$ such that $g=h\circ f$.
    \end{enumerate}
\end{proposition}

\begin{proof}
See Lemma 1.14, p. 18, in \cite{Kallenberg2021}
\end{proof}

Using the Doob-Dynkin Lemma, it is straightforward to show that a statistic $T:\mathcal{X}\to (\mathcal{T},\Sigma _\mathcal{T})$ is $\mathcal{E}$-sufficient iff $\sigma (T)$ is $\mathcal{E}$-sufficient. This equivalence generally does not extend to minimal sufficiency; as demonstrated in Examples~9.7 and 9.8 of \cite{Heyer1982}, $T$ may be minimal sufficient while $\sigma (T)$ is not, and conversely. The following proposition establishes a sufficient assumption on the statistic's codomain under which this discrepancy disappears in one direction: if the codomain is standard Borel (see Definition \ref{def:stand_Borel}), then minimality at the $\sigma$-algebra level lifts to minimality at the statistic level.

\begin{proposition}
\label{prop:4.5}
Let $\mathcal{E}:=(\mathcal{X},\Sigma,\{P_\theta\}_{\theta\in\Theta})$ be a statistical model and $T:\mathcal{X}\to (\mathcal{T},\Sigma _\mathcal{T})$. Suppose that $(\mathcal{T},\Sigma _\mathcal{T})$ is a standard Borel space. If $\sigma (T)$ is minimal sufficient, then $T$ is minimal sufficient.
\end{proposition}

\begin{proof}
Suppose that $\sigma (T)$ is minimal sufficient. Then $\sigma (T)$ is sufficient, implying that $T$ is sufficient. Let $S:\mathcal{X}\to (\mathcal{S},\Sigma _\mathcal{S})$ be an arbitrary sufficient statistic. Then $\sigma (S)$ is sufficient, which implies, using the hypothesis, that $\sigma (T)\subseteq \overline{\sigma (S)}^\mathcal{E}$. Therefore, $T$ is $\overline{\sigma (S)}^\mathcal{E}$-measurable, which implies, using Proposition \ref{prop:4.1}, that there exists a $\sigma (S)$-measurable function $g:\mathcal{X}\to \mathcal{T}$ such that, for each $\theta\in\Theta$, we have $T=g$ $P_\theta$-a.e.

Since $g$ is $\sigma (S)$-measurable, we have $\sigma (g)\subseteq \sigma (S)$, which implies, using the Doob-Dynkin Lemma (see Proposition \ref{doob_lemma}), that there exists a $(\Sigma _\mathcal{S},\Sigma _\mathcal{T})$-measurable function $f:\mathcal{S}\to \mathcal{T}$ such that $g=f(S)$, implying that, for each $\theta\in\Theta$, we have $T=f(S)$ $P_\theta$-a.e.
\end{proof}

\begin{proposition}
\label{prop:4.7}
Let $\mathcal{E}:=(\mathcal{X},\Sigma,\{P_\theta\}_{\theta\in\Theta})$ be a statistical model and $(c_n)_{n\in\mathbb{N}}$ a sequence of $(0,\infty)$ such that $\sum _{n\in\mathbb{N}}c_n=1$. If $\{\theta _n\}_{n\in\mathbb{N}}\subseteq\Theta$ is dense in $(\Theta,\tau \langle \mathcal{E}\rangle )$, then the following statements are true:
    \begin{enumerate}[label=\roman*.]
    \item $P:\Sigma \to \mathbb{R}$ given by $P:=\sum _{n\in\mathbb{N}}c_nP_{\theta _n}$ is a probability measure such that $\mathcal{E}\equiv P$;
    \item For each $n\in\mathbb{N}$, choose $f_{\theta _n}\in [dP_{\theta _n}/dP]$. Then $\mathfrak{A}:=\vee _{n\in\mathbb{N}}\sigma (f_{\theta _n})$ is a countably generated minimal $\mathcal{E}$-sufficient sub-$\sigma$-algebra of $\Sigma$.
    \end{enumerate}
\end{proposition}

\begin{proof}
Suppose that $\{\theta _n\}_{n\in\mathbb{N}}\subseteq\Theta$ is dense in $(\Theta,\tau \langle \mathcal{E}\rangle )$. Define $P:\Sigma \to \mathbb{R}$ by $P:=\sum _{n\in\mathbb{N}}c_nP_{\theta _n}$. Then $P$ is a probability measure. Repeating an argument from Lemma 4.3 of \cite{Strasser1985}, we conclude that $\mathcal{E}\ll P$, which implies that $\mathcal{E}\equiv P$ since $P\ll \{P_{\theta_n} \}_{n\in\mathbb{N}}$. Hence, $\mathcal{E}$ is dominated.

For each $\theta\in\Theta$, choose $f_\theta \in [dP_\theta /dP]$. Define $\mathfrak{A}:=\vee _{n\in\mathbb{N}}\sigma (f_{\theta _n})$.

Repeating an argument from that lemma, we conclude that $\mathfrak{A}$ is a countably generated sub-$\sigma$-algebra of $\Sigma$  and that $f_\theta =\mathbb{E}_P[f_\theta |\mathfrak{A}]$ $P$-a.e. for every $\theta\in\Theta $, allowing us to conclude that $\mathfrak{A}$ is sufficient \citep[Lemma 20.6]{Strasser1985}. Hence, it remains to prove that $\mathfrak{A}$ is minimal.

From the proof of Theorem 8.8 in \cite{Heyer1982}, we know $\mathfrak{B}:=\vee_{\theta\in\Theta }\sigma (f_\theta )$ is minimal sufficient, which allows us to conclude that $\mathfrak{A}$ is minimal since $\mathfrak{A}\subseteq \mathfrak{B}$.

A more detailed proof can be found in Proposition 2.5.5 of \cite{Cavalcante2026}.
\end{proof}

We now state a result similar to the Doob-Dynkin Lemma that will be employed to prove our proposed methods.

\begin{proposition}
\label{prop:4.8}
Let $(\mathcal{X},\Sigma _\mathcal{X})$, $(\mathcal{Y},\Sigma _\mathcal{Y})$, and $(\mathcal{Z},\Sigma _\mathcal{Z})$ be three measurable spaces and $f:\mathcal{X}\to \mathcal{Y}$ and $g:\mathcal{X}\to \mathcal{Z}$ two measurable functions. Suppose that $(\mathcal{X},\Sigma _\mathcal{X})$ is an analytic Borel space and $(\mathcal{Y},\Sigma _\mathcal{Y})$ is separable. If $(\forall x,y\in \mathcal{X})\big(f(x)=f(y)\Rightarrow g(x)=g(y)\big)$, then $\sigma (g)\subseteq \sigma (f)$.
\end{proposition}

\begin{proof}
Suppose that $(\forall x,y\in \mathcal{X})\big(f(x)=f(y)\Rightarrow g(x)=g(y)\big)$. Then $h:f[\mathcal{X}]\to \mathcal{Z}$ given by $h(f(x)):=g(x)$ is well-defined.

	Let $E\in \sigma (g)$ be an arbitrary element. Then there exists $F\in \Sigma _\mathcal{Z}$ such that $E=g^{-1}[F]$, allowing us to conclude that 

\begin{equation}
\label{eq:4.8.1}
E=(h\circ f)^{-1}[F]=f^{-1}[h^{-1}[F]]=\cup _{y\in h^{-1}[F]}f^{-1}[\{y\}].
\end{equation}

	By hypothesis, $(\mathcal{Y},\Sigma _\mathcal{Y})$ is countably generated; therefore, $\sigma (f)$ is countably generated (since if $\Sigma _\mathcal{Y}$ is generated by $\mathcal{S}\subseteq 2^\mathcal{Y}$, then $\sigma (f)$ is generated by $f^{-1}(\mathcal{S})$). Moreover, it is straightforward to verify that $f^{-1}[\{y\}]$ is an atom of $\sigma (f):=f^{-1}(\Sigma _\mathcal{Y})$ for every $y\in f[\mathcal{X}]$ since, by hypothesis, $\Sigma _\mathcal{Y}$ is separable \citep[see Section 8.6 of][for the definition of an atom]{Cohn2013}.

	As $f$ is measurable, it follows that $\sigma (f)$ is a sub-$\sigma$-algebra of $\Sigma _\mathcal{X}$ and $E\in \Sigma _\mathcal{X}$. Recall that if $(\mathcal{W},\Sigma )$ is a separable measurable space, then $(A,\Sigma |_A)$ is separable for every subset $A\subseteq\mathcal{W}$. 
    Hence, by Proposition 8.6.5 of \cite{Cohn2013}, our definition of an analytic Borel space coincides with the definition of an analytic space given in Section 8.6 of \cite{Cohn2013}. Therefore, using (\ref{eq:4.8.1}) and Theorem 8.6.7 of \cite{Cohn2013}, we conclude $E\in \sigma (f)$.
\end{proof}

\begin{proposition}
\label{prop:4.9}
Let $(\mathcal{X},\Sigma _\mathcal{X})$, $(\mathcal{Y},\Sigma _\mathcal{Y})$, and $(\mathcal{Z},\Sigma _\mathcal{Z})$ be three measurable spaces and $f:\mathcal{X}\to \mathcal{Y}$ and $g:\mathcal{X}\to \mathcal{Z}$ two measurable functions. If $(\mathcal{X},\Sigma _\mathcal{X})$ is an analytic Borel space, $(\mathcal{Y},\Sigma _\mathcal{Y})$ is separable, and $(\mathcal{Z},\Sigma _\mathcal{Z})$ is a standard Borel space, then the following statements are \underline{equivalent}:
    \begin{enumerate}
        \item $\sigma (g)\subseteq \sigma (f)$;
        \item There exists a measurable function $h:\mathcal{Y}\to \mathcal{Z}$ such that $g=h\circ f$;
        \item There exists a (not necessarily measurable) function $\tilde h:\mathcal{Y}\to \mathcal{Z}$ such that $g=\tilde h\circ f$;
        \item $(\forall x,y\in \mathcal{X})\big(f(x)=f(y)\Rightarrow g(x)=g(y)\big)$.
    \end{enumerate}
\end{proposition}

\begin{proof}
 The equivalence $1.\Leftrightarrow 2.$ is a direct consequence of the Doob-Dynkin Lemma. The implications $2.\Rightarrow 3.$ and $3.\Rightarrow 4.$ are trivial. The implication $4.\Rightarrow 1.$ is a direct consequence of Proposition \ref{prop:4.8}.
\end{proof}

The next corollary shows that the conditions required by Proposition \ref{prop:4.9} can be relaxed if we do not require equality at every point, which is a condition that we do not need to prove our methods. Before we prove it, we need the following lemma.

In the next lemma, $\mathfrak{B}_\mathbb{R}$ denotes the Borel $\sigma$-algebra of $\mathbb{R}$.

\begin{lemma}
\label{lemma_doob}
Let $(\mathcal{X},\Sigma,\mathbb{P})$ be a probability space with $(\mathcal{X},\Sigma)$ being an analytic Borel space, and let $(\mathcal{Y},\Sigma_\mathcal{Y})$ be a separable measurable space and $f:\mathcal{X}\to\mathcal{Y}$ a measurable function. If $\alpha:\mathcal{Y}\to\mathbb{R}$ is injective and $\Sigma_\mathcal{Y} =\alpha^{-1}(\mathfrak{B}_\mathbb{R})$, then there exists a measurable function $\beta:\mathbb{R}\to\mathcal{Y}$ such that $\beta\circ \alpha \circ f=f$ $\mathbb{P}$-a.e.
\end{lemma}

\begin{proof}

Suppose that $\alpha:\mathcal{Y}\to \mathbb{R}$ is an injective function such that $\Sigma_\mathcal{Y}=\alpha^{-1}(\mathfrak{B}_\mathbb{R})$.

In the proof of Proposition \ref{prop:4.8}, we demonstrated that our definition of analytic Borel spaces coincides with the definition of analytic spaces given by \cite{Cohn2013}. Therefore, using Corollary 8.4.3 and Lemma 8.6.1 of \cite{Cohn2013}, we can conclude that there exists $B\in\mathfrak{B}_\mathbb{R}$ such that $B\subseteq (\alpha\circ f)[\mathcal{X}]$ and $\mathbb{P}((\alpha\circ f)^{-1}[B])=1$.

Choose $y_0\in \mathcal{Y}$ and define $\beta:\mathbb{R}\to \mathcal{Y}$ by $\beta(t):=\begin{cases}\alpha^{-1}(t),&t\in B\\y_0,&t\notin B\end{cases}$.

Next, we will prove that $\beta$ is $(\mathfrak{B}_\mathbb{R},\Sigma_\mathcal{Y})$-measurable.

Let $E\in\Sigma_\mathcal{Y}$ be arbitrary. From the hypothesis, we know that there exists $A\in\mathfrak{B}_\mathbb{R}$ such that $E=\alpha^{-1}[A]$, allowing us to conclude that 
\begin{align}
\beta^{-1}[E]=&(\beta^{-1}[E]\cap B)\cup (\beta^{-1}[E]\cap B^c)=(\alpha[\alpha^{-1}[A]]\cap B)\cup (\beta^{-1}[E]\cap B^{c})\\=&(A\cap \alpha[\mathcal{Y}]\cap B)\cup (\beta^{-1}[E]\cap B^c)=(A\cap B)\cup (\beta^{-1}[E]\cap B^c)
\end{align}

Therefore, since $A,B\in \mathfrak{B}_\mathbb{R}$ and $\beta$ is constant on $B^c$, we infer that $\beta^{-1}[E]\in \mathfrak{B}_\mathbb{R}$, implying that $\beta$ is $(\mathfrak{B}_\mathbb{R},\Sigma_\mathcal{Y})$-measurable.

Since $(\alpha\circ f)(x)\in B$ for every $x\in (\alpha\circ f)^{-1}[B]$, we have $\beta(\alpha(f(x)))=f(x)$ for every $x\in (\alpha\circ f)^{-1}[B]$, implying that $\beta\circ \alpha\circ f=f$ $\mathbb{P}$-a.e., because $\mathbb{P}((\alpha\circ f)^{-1}[B])=1$.

\end{proof}

\begin{corollary}
\label{cor_doob}
Let $(\mathcal{X},\Sigma,\mathbb{P})$ be a probability space with $(\mathcal{X},\Sigma)$ being an analytic Borel space, and let $(\mathcal{Y},\Sigma _\mathcal{Y})$ and $(\mathcal{Z},\Sigma _\mathcal{Z})$ be separable measurable spaces. Suppose that $f:\mathcal{X}\to \mathcal{Y}$ and $g:\mathcal{X}\to \mathcal{Z}$ are measurable functions. If $(\forall x,y\in \mathcal{X})\big(f(x)=f(y)\Rightarrow g(x)=g(y)\big)$, then there exists a measurable function $h:\mathcal{Y}\to \mathcal{Z}$ such that $g=h\circ f$ $\mathbb{P}$-a.e.
\end{corollary}

\begin{proof}
Suppose that $(\forall x,y\in \mathcal{X})\big(f(x)=f(y)\Rightarrow g(x)=g(y)\big)$.

By hypothesis, $(\mathcal{Z},\Sigma _\mathcal{Z})$ is a separable measurable space; therefore, by Corollary D.7.3.III of \cite{Cavalcante2026}, there exists an injective measurable function $\alpha:\mathcal{Z}\to\mathbb{R}$ such that $\Sigma_\mathcal{Z}=\alpha ^{-1}(\mathfrak{B}_\mathbb{R})$.

The hypothesis implies that $(\forall x,y\in \mathcal{X})(f(x)=f(y)\Rightarrow (\alpha \circ g)(x)=(\alpha \circ g)(y)\big)$, allowing us to conclude, using Proposition \ref{prop:4.9}, that there exists a measurable function $h':\mathcal{Y}\to \mathbb{R}$ such that 
\begin{equation}
\label{eq_doob}
\alpha \circ g=h'\circ f
\end{equation}

By Lemma \ref{lemma_doob}, there exists a measurable function $\beta:\mathbb{R}\to \mathcal{Z}$ such that $\beta\circ \alpha\circ g=g$ $\mathbb{P}$-a.e. Using (\ref{eq_doob}), this implies that $g=h\circ f$ $\mathbb{P}$-a.e., where $h:\mathcal{Y}\to \mathcal{Z}$ given by $h:=\beta \circ h'$ is measurable.
\end{proof}

We now proceed to the proofs of the methods presented in Section 3.

\begin{lemma}
\label{lemma:4.1}
Let $\mathcal{E}:=(\mathcal{X},\Sigma,\{P_\theta\}_{\theta\in\Theta})$ be a statistical model, $\mu :\Sigma\to\overline{\mathbb{R}}$ a $\sigma$-finite measure such that $\mathcal{E}\ll \mu$, and $T:\mathcal{X}\to (\mathcal{T},\Sigma _\mathcal{T})$. Suppose that $(\mathcal{X},\Sigma )$ is an analytic Borel space and $(\mathcal{T},\Sigma _\mathcal{T})$ is separable. For each $\theta\in\Theta$, choose $f_\theta \in [dP_\theta /d\mu ]$. For any $x\in \mathcal{X}$ and $\Theta'\subseteq\Theta$, define

$$D(x,\Theta '):=\big\{y\in \mathcal{X}:(\exists h_{xy}\in (0,\infty ))(\forall \theta \in \Theta ')\big(f_\theta (y)=f_\theta (x)h_{xy}\big)\big\}.$$

If there exists a non-empty countable subset $\Theta _0\subseteq\Theta$ that is dense in $(\Theta,\tau \langle \mathcal{E}\rangle )$ such that, for any $x,y\in\mathcal{X}$ satisfying $T(x)=T(y)$, we have $y\in D(x,\Theta _0)$, then $T$ is $\mathcal{E}$-sufficient.
\end{lemma}

\begin{proof}
Suppose that there exists a non-empty countable subset $\Theta _0\subseteq\Theta$ that is dense in $(\Theta,\tau \langle \mathcal{E}\rangle )$ such that, for any $x,y\in\mathcal{X}$ satisfying $T(x)=T(y)$, we have $y\in D(x,\Theta _0)$. 

Suppose that $\Theta _0=\{\theta _n\}_{n\in\mathbb{N}}$. From Proposition \ref{prop:4.7}.i, we know that $P:\Sigma\to \mathbb{R}$ given by $P:=\sum _{n\in\mathbb{N}}2^{-(n+1)}P_{\theta _n}$ is a probability measure such that $\mathcal{E}\equiv P$.

By hypothesis, $\mathcal{E}\ll \mu$; therefore, $P\ll \mu$. Moreover, due to the almost everywhere uniqueness of a Radon-Nikodym derivative, we can conclude that $dP/d\mu =\sum _{n\in\mathbb{N} }2^{-(n+1)}f_{\theta _n}$ $\mu$-a.e.

Define $f:\mathcal{X}\to \overline{\mathbb{R}}$ by $f:=\sum _{n\in\mathbb{N}}2^{-(n+1)}f_{\theta _n}$. Then $f$ is a non-negative measurable function. Let $n\in\mathbb{N}$ be an arbitrary element. We have $\frac{dP_{\theta_n}}{d\mu }=\frac{dP_{\theta _n}}{dP}\frac{dP}{d\mu }$, which implies that $f_{\theta _n}=\frac{dP_{\theta _n}}{dP}f$ $\mu$-a.e.

Observe that $P(f\in \{0,\,\infty \})=0$, which implies that $P(f\in (0,\infty ))=1$.

Define $g_{\theta _n}:\mathcal{X}\to\mathbb{R}$ by $g_{\theta _n}(x):=\begin{cases}\frac{f_{\theta _n}(x)}{f(x)},&f(x)\in (0,\infty )\\0,&f(x)\in\{0,\infty \}\end{cases}$.

Then $g_{\theta _n}$ is a non-negative measurable function. Moreover, we have $g_{\theta _n}=\frac{dP_{\theta _n}}{dP}$ $P$-a.e., which implies that $g_{\theta _n}\in [dP_{\theta _n}/dP]$.

The previous results hold for every $n\in\mathbb{N}$ since, by hypothesis, $n\in\mathbb{N}$ is arbitrary.

Define $S:\mathcal{X}\to \mathbb{R}^\mathbb{N}$ by $S(x)(n):=g_{\theta _n}(x)$. Then $S$ is measurable. Furthermore, it is straightforward to verify that $\sigma (S)=\vee _{n\in\mathbb{N}}\sigma (g_{\theta _n})$, allowing us to conclude, using Proposition \ref{prop:4.7}.ii, that $\sigma (S)$ is $\mathcal{E}$-sufficient. Next, we will prove that, for any $x,y\in \mathcal{X}$ satisfying $T(x)=T(y)$, we have $S(x)=S(y)$.

Let $x,y\in\mathcal{X}$ be arbitrary elements such that $T(x)=T(y)$. Then, using the hypothesis, we obtain $y\in D(x,\Theta _0)$, implying that there exists $h_{xy}\in (0,\infty )$ such that, for each $n\in\mathbb{N}$, we have $f_{\theta _n}(y)=f_{\theta_n} (x)h_{xy}$. Hence, using the definition of $f$, we conclude that $f(y)=f(x)h_{xy}$.

If $f(x)\in \{0,\infty \}$, then $f(y)\in \{0,\infty \}$, allowing us to conclude that $S(x)(n)=0=S(y)(n)$ for every $n\in\mathbb{N}$ and, consequently, that $S(x)=S(y)$. Suppose that $f(x)\notin \{0,\infty \}$. Then $f(y)\notin\{0,\infty \}$, which implies that for each $n\in\mathbb{N}$ we have $$S(y)(n)=\frac{f_{\theta _n}(y)}{f(y)}=\frac{f_{\theta _n}(x)h_{xy}}{f(x)h_{xy}}=\frac{f_{\theta _n}(x)}{f(x)}=S(x)(n)$$ and, consequently, $S(x)=S(y)$.

Therefore, for any $x,y\in \mathcal{X}$ satisfying $T(x)=T(y)$, we have $S(x)=S(y)$, which implies, using Proposition \ref{prop:4.8} and the hypotheses that $(\mathcal{X},\Sigma )$ is analytic Borel and $(\mathcal{T},\Sigma _\mathcal{T})$ is separable, that $\sigma (S)\subseteq \sigma (T)$. Hence, recalling that we proved that $S$ is $\mathcal{E}$-sufficient, we conclude, using Corollary 8.5 in  \cite{Heyer1982}, that $\sigma (T)$ is $\mathcal{E}$-sufficient and, consequently, that $T$ is $\mathcal{E}$-sufficient.
\end{proof}

The next corollary is an immediate consequence of Corollary \ref{cor:4.1} and Lemma~\ref{lemma:4.1}, with
$(\Theta,\tau\langle\mathcal E\rangle)$ replaced by $(\Theta,d_{\mathcal E})$.

\begin{corollary}
\label{cor:4.2}
Let $\mathcal{E}:=(\mathcal{X},\Sigma,\{P_\theta\}_{\theta\in\Theta})$ be a statistical model, $\mu :\Sigma\to\overline{\mathbb{R}}$ a $\sigma$-finite measure such that $\mathcal{E}\ll \mu$, and $T:\mathcal{X}\to (\mathcal{T},\Sigma _\mathcal{T})$. Suppose that $(\mathcal{X},\Sigma )$ is an analytic Borel space and $(\mathcal{T},\Sigma _\mathcal{T})$ is separable. For each $\theta\in\Theta$, choose $f_\theta \in [dP_\theta /d\mu ]$. For any $x\in \mathcal{X}$ and $\Theta'\subseteq\Theta$, define $D(x,\Theta '):=\big\{y\in \mathcal{X}:(\exists h_{xy}\in (0,\infty ))(\forall \theta \in \Theta ')\big(f_\theta (y)=f_\theta (x)h_{xy}\big)\big\}$. If there exists a non-empty countable subset $\Theta _0\subseteq\Theta$ that is dense in $(\Theta,d_\mathcal{E})$ such that, for any $x,y\in\mathcal{X}$ satisfying $T(x)=T(y)$, we have $y\in D(x,\Theta _0)$, then $T$ is $\mathcal{E}$-sufficient.
\end{corollary}
\begin{proof}
    By Corollary \ref{cor:4.1}, $\Theta_0$ dense in $(\Theta,d_{\mathcal E})$ implies $\Theta_0$ dense in $(\Theta,\tau\langle\mathcal E\rangle)$, so the conclusion follows directly from Lemma~\ref{lemma:4.1}.
\end{proof}

Finally, we provide the proof of Method \ref{meth:3.1}.

\begin{proof}[Proof of Method \ref{meth:3.1}]
Suppose that $T$ is $\mathcal{E}$-sufficient and there exists a non-empty countable set $\Theta _0\subseteq\Theta$ such that, for any $x,y\in\mathcal{X}$ satisfying $y\in D(x,\Theta _0)$, we have $T(x)=T(y)$.

From Proposition \ref{prop:4.2}, we know that there exists a non-empty countable subset $\Theta _{1}\subseteq \Theta$ such that $\mathcal{E}\equiv \{P_\theta \}_{\theta\in\Theta _{1}}$. Define $\Theta _2:=\Theta _0\cup \Theta _{1}$. Then $\Theta _2$ is a non-empty countable subset of $\Theta$. Suppose $\Theta _2=\{\theta _n\}_{n\in\mathbb{N}}$. Denote by $\overline{\Theta _2}$ the closure of $\Theta _2$ in $\Theta$ w.r.t. the pseudo-metric $d_\mathcal{E}$. Then $\mathcal{E}\equiv \{P_\theta \}_{\theta\in\overline{\Theta _2}}$, since $\Theta _1\subseteq\overline{\Theta _2}$.

Consider the statistical model $\mathcal{F}:=((\mathcal{X},\Sigma ),\{P_\theta\}_{\theta\in\overline{\Theta_2}})$. It is straightforward to verify that $\Theta _2$ is dense in $(\overline{\Theta _2},d_\mathcal{F})$ (because $d_\mathcal{F}$ is a restriction of $d_\mathcal{E}$), which implies, using Corollary \ref{cor:4.1}, that $\Theta _2$ is dense in $(\overline{\Theta _2},\tau \langle \mathcal{F}\rangle )$. Therefore, using Proposition \ref{prop:4.7}.i, we conclude that $P:\Sigma\to\mathbb{R}$ given by $P:=\sum _{n\in\mathbb{N}}2^{-(n+1)}P_{\theta _n}$ is a probability measure satisfying $\mathcal{F}\equiv P$.

By hypothesis, $\mathcal{E}\ll \mu$; therefore, $\mathcal{F}\ll \mu $, which implies $P\ll \mu $. Moreover, due to the almost everywhere uniqueness of a Radon-Nikodym derivative, we can conclude that  $dP/d\mu =\sum _{n\in\mathbb{N} }2^{-(n+1)}f_{\theta _n} $ $\mu$-a.e.

Define $f:\mathcal{X}\to \overline{\mathbb{R}}$ by $f:=\sum _{n\in\mathbb{N}}2^{-(n+1)}f_{\theta _n}$. Then $f$ is a non-negative measurable function. Let $n\in\mathbb{N}$ be arbitrary. We have $\frac{dP_{\theta_n}}{d\mu }=\frac{dP_{\theta _n}}{dP}\frac{dP}{d\mu }$ $\mu$-a.e., which implies that $f_{\theta _n}=\frac{dP_{\theta _n}}{dP}f$ $\mu$-a.e.

Observe that $P(f\in \{0,\,\infty \})=0$, which implies $P(f\in (0,\infty ))=1$.

Define $g_{\theta _n}:\mathcal{X}\to \mathbb{R}$ by $g_{\theta _n}(x):=\begin{cases}\frac{f_{\theta _n}(x)}{f(x)},&f(x)\in (0,\infty )\\f_{\theta _n}(x),&f(x)\in\{0,\infty \}\end{cases}$. Then $g_{\theta _n}$ is a non-negative measurable function. Moreover, we have $g_{\theta _n}=\frac{dP_{\theta _n}}{dP}$ $P$-a.e., which implies that $g_{\theta _n}\in [dP_{\theta _n}/dP]$.

The previous results hold for every $n\in\mathbb{N}$ since, by hypothesis, $n\in\mathbb{N}$ is arbitrary.

Define $S:\mathcal{X}\to \mathbb{R}^\mathbb{N}$ by $S(x)(n):=g_{\theta _n}(x)$. Then $S$ is measurable. Furthermore, it is straightforward to prove that $\sigma (S)=\vee _{n\in\mathbb{N}}\sigma (g_{\theta _n})$, allowing us to conclude, using Propositions \ref{prop:4.5} and \ref{prop:4.7}.ii, that $S$ is minimal $\mathcal{F}$-sufficient since $\{\theta _n\}_{n\in\mathbb{N}}$ is dense in $(\overline{\Theta _2},\tau \langle \mathcal{F}\rangle )$. Next, we will prove that, for any $x,y\in\mathcal{X}$ satisfying $S(x)=S(y)$, we have $T(x)=T(y)$.

Let $x,y\in\mathcal{X}$ be arbitrary elements satisfying $S(x)=S(y)$.

First, suppose that $f(x)\in \{0,\infty \}$. If $f(y)\notin \{0,\infty \}$, then $$f_{\theta _n}(x)=S(x)(n)=S(y)(n)=\frac{f_{\theta _n}(y)}{f(y)}$$ for every $n\in\mathbb{N}$, which implies, recalling the definition of $f$, that $$f(x)=\sum _{n\in\mathbb{N}}2^{-(n+1)}f_{\theta _n}(x)=\frac{\sum _{n\in\mathbb{N}}2^{-(n+1)}f_{\theta _n}(y)}{f(y)}=1,$$ contradicting the assumption $f(x)\in \{0,\infty \}$. Hence, $f(y)\in \{0,\infty \}$, which implies that, for each $n\in\mathbb{N}$, we have $f_{\theta _n}(x)=S(x)(n)=S(y)(n)=f_{\theta _n}(y)$ and, consequently, $y\in D(x,\Theta _0)$ (because $\Theta _0\subseteq\Theta _2$). Therefore, using the hypothesis, we infer that $T(x)=T(y)$.

Now suppose that $f(x)\notin \{0,\infty \}$. Repeating the previous argument, we can conclude that $f(y)\notin \{0,\infty \}$, which implies that $\frac{f_{\theta _n}(x)}{f(x)}=S(x)(n)=S(y)(n)=\frac{f_{\theta _n}(y)}{f(y)}$ for each $n\in\mathbb{N}$, implying that $f_{\theta _n}(y)=f_{\theta _n}(x)h_{xy}$ for every $n\in\mathbb{N}$, where $h_{xy}:=\frac{f(y)}{f(x)}>0$. Therefore, we obtain $y\in D(x,\Theta _0)$ (because $\Theta _0\subseteq\Theta _2$), which implies, using the hypothesis, that $T(x)=T(y)$.

From the previous cases, we conclude that $T(x)=T(y)$.

Therefore, for any $x,y\in\mathcal{X}$ satisfying $S(x)=S(y)$, we have $T(x)=T(y)$, which implies, using Corollary \ref{cor_doob}, that there exists a measurable function $h:\mathbb{R}^\mathbb{N}\to \mathcal{T}$ such that $T=h(S)$ $P$-a.e. Hence, $T=h(S)$ $P_\theta$-a.e. for every $\theta\in\overline{\Theta_2}$.

By hypothesis, $T$ is $\mathcal{E}$-sufficient; therefore, $T$ is $\mathcal{F}$-sufficient, allowing us to conclude, using the previous a.e. equality and the fact that $S$ is minimal $\mathcal{F}$-sufficient, that $T$ is minimal $\mathcal{F}$-sufficient. Hence, $T$ is $\mathcal{E}$-sufficient and minimal $\{P_\theta\}_{\theta\in\overline{\Theta _2}}$-sufficient, allowing us to conclude, using the definition of minimal sufficiency of statistics and recalling that $\mathcal{E}\equiv \{P_\theta \}_{\theta\in\overline{\Theta _2}}$ , that $T$ is minimal $\mathcal{E}$-sufficient.
\end{proof}

Before proving Method \ref{meth:3.2}, we establish the following lemma.

\begin{lemma}
\label{lemma:4.2}
Let $\mathcal{E}:=(\mathcal{X},\Sigma,\{P_\theta\}_{\theta\in\Theta})$ be a statistical model and $\mu :\Sigma\to\overline{\mathbb{R}}$ a $\sigma$-finite measure such that $\mathcal{E}\ll \mu$. For each $\theta\in\Theta$, choose $f_\theta \in [dP_\theta /d\mu ]$. If $\Theta _0\subseteq \Theta$ is a subset such that, for each $\theta\in\Theta$, there exists a sequence $(\theta _n)_{n\in\mathbb{N}}$ of $\Theta _0$ satisfying $f_\theta =\lim_{n\to\infty}f_{\theta _n}$ $\mu$-a.e., then $\Theta _0$ is dense in $(\Theta,d_\mathcal{E})$.
\end{lemma}

\begin{proof}
Suppose that $\Theta _0\subseteq \Theta$ is a subset such that, for each $\theta\in\Theta$, there exists a sequence $(\theta _n)_{n\in\mathbb{N}}$ of $\Theta _0$ satisfying $f_\theta =\lim_{n\to\infty}f_{\theta _n}$ $\mu$-a.e. Let $\theta\in\Theta $ and $\varepsilon >0$ be arbitrary elements. To conclude the proof, it suffices to prove that there exists $\theta '\in\Theta _0$ such that $d_\mathcal{E}(\theta ,\theta ')<\varepsilon $.

From the hypothesis, we know that there exists a sequence $(\theta _n)_{n\in\mathbb{N}}$ of $\Theta _0$ such that $f_\theta =\lim_{n\to\infty}f_{\theta _n}$ $\mu$-a.e. Since $\int _\mathcal{X}f_{\theta _n}d\mu =\int _\mathcal{X}f_\theta d\mu =1<\infty $ for every $n\in\mathbb{N}$, we can conclude, using Proposition 1.7.8 of \cite{Lerner2014}, that $\lim_{n\to\infty}\int _\mathcal{X}|f_{\theta _n}-f_\theta |d\mu =0$, which implies, using Proposition \ref{prop:4.3}, that $\lim_{n\to\infty}d_\mathcal{E}(\theta _n,\theta )=0$ and, consequently, that there exists $N\in\mathbb{N}$ such that $d_\mathcal{E}(\theta _N,\theta )<\varepsilon $.
\end{proof}

\begin{proof}[Proof of Method \ref{meth:3.2}]
Suppose that, for any $x,y\in\mathcal{X}$, we have $T(x)=T(y)$ iff $y\in D(x)$.

Define $D(x,\Theta _0):=\big\{y\in \mathcal{X}:(\exists h_{xy}\in (0,\infty ))(\forall \theta \in \Theta_0 )\big(f_\theta (y)=f_\theta (x)h_{xy}\big)\big\}$. Next, we will prove that, for any $x,y\in\mathcal{X}$, we have $T(x)=T(y)$ iff $y\in D(x,\Theta _0)$.

$(\Rightarrow )$ Suppose that $T(x)=T(y)$. Then, using the hypothesis, we conclude that $y\in D(x)$, which implies that $y\in D(x,\Theta _0)$, since $D(x)\subseteq D(x,\Theta _0)$. $\blacksquare$

$(\Leftarrow )$ Suppose that $y\in D(x,\Theta _0)$. Then there exists $h_{xy}\in (0,\infty )$ such that $f_\theta (y)=f_\theta (x)h_{xy}$ for every $\theta\in\Theta _0$. Let $\theta\in\Theta$ be an arbitrary element. From the hypotheses, we know that there exists a sequence $(\theta _n)_{n\in\mathbb{N}}$ of $\Theta _0$ such that the limit $\lim_{n\to\infty}f_{\theta_n}$ exists $\mu$-a.e. and for every $z\in\mathcal{X}$ we have 

\begin{equation}
\label{eq_3.2}
f_\theta (z)=\begin{cases}\lim_{n\to \infty}f_{\theta _n}(z),&\text{if the limit exists,}\\0,&\text{otherwise}.\end{cases}
\end{equation}

Since $f_{\theta_n}(y)=f_{\theta_n}(x)h_{xy}$ for every $n\in\mathbb{N}$ and $h_{xy}\in(0,\infty)$, we can infer that $\lim_{n\to\infty}f_{\theta_n}(x)$ exists if and only if $\lim_{n\to\infty}f_{\theta_n}(y)$ exists, allowing us to conclude, using (\ref{eq_3.2}), that $f_\theta(y)=f_\theta(x)h_{xy}$. As $\theta\in\Theta$ was arbitrary, it follows that $f_\theta(y)=f_\theta(x)h_{xy}$ for every $\theta\in\Theta$, hence $y\in D(x)$. Therefore, by the hypothesis $T(x)=T(y)$.
$\blacksquare$

From the previous implications, we conclude that, for any $x,y\in\mathcal{X}$, we have $T(x)=T(y)$ iff $y\in D(x,\Theta _0)$, which allows us to conclude, using Lemmas \ref{lemma:4.1} and \ref{lemma:4.2} and Method \ref{meth:3.1}, that $T$ is minimal $\mathcal{E}$-sufficient.
\end{proof}

We conclude with the proof of Method 3.3 taken from Proposition 1.6.9 in \cite{Pfanzagl1994}.

\begin{proof}[Proof of Method \ref{meth:3.3}]
Suppose that, for any $a_0,\cdots,a_k\in\mathbb{R}$ satisfying $\sum _{i=1}^ka_i\eta_i (\theta )=a_0$ for every $\theta\in\Theta$, we have $a_i=0$ for every $i\in \{0,\cdots ,k\}$. Next, we will use Method \ref{meth:3.1} to obtain the desired result.

Define $T:\mathcal{X}\to \mathbb{R}^k$ by $T(x):=(T_1(x),\cdots, T_k(x))$. Then $T$ is measurable. Define $\eta :\Theta \to \mathbb{R}^k$ by $\eta (\theta ):=(\eta _1(\theta ),\cdots,\eta _k(\theta ))$.

Denote by $\langle \cdot ,\cdot \rangle $ the canonical inner product of $\mathbb{R}^k$. That is, $\langle a,b\rangle :=\sum _{i=1}^ka_ib_i$ for any $a:=(a_1,\cdots, a_k),b:=(b_1,\cdots,b_k)\in\mathbb{R}^k$. From the hypothesis, we know that, for each $\theta\in\Theta$, we have $\frac{dP_\theta }{d\mu }(x)=e^{\langle \eta (\theta ),T(x)\rangle -B(\theta )}h(x)$ $\mu$-a.e. Therefore, using the Neyman-Fisher Factorization Theorem, we conclude that $T$ is sufficient.

Define $\nu :\Sigma \to \overline {\mathbb{R}}$ by $\nu (E):=\int _Eh(x)d\mu (x)$. Then $\nu$ is a $\sigma$-finite measure. Moreover, for any $E\in \Sigma$ and $\theta\in\Theta$, we have $$P_\theta (E)=\int _Ee^{\langle \eta (\theta ),T(x)\rangle -B(\theta )}h(x)d\mu (x)=\int _Ee^{\langle \eta (\theta ),T(x)\rangle-B(\theta )}d\nu (x),$$ which allows us to conclude that $\mathcal{E}\ll \nu$ and 

\begin{equation}
\label{eqmeth:3.3.1}
\frac{dP_\theta }{d\nu }(x)=e^{\langle \eta (\theta ),T(x)\rangle -B(\theta )}\,\,\nu\text{-a.e.}
\end{equation}

For each $\theta\in\Theta$, define $f_\theta :\mathcal{X}\to \mathbb{R}$ by $f_\theta (x):=e^{\langle\eta (\theta),T(x)\rangle -B(\theta ) }$. From (\ref{eqmeth:3.3.1}), we know that $f_\theta \in [dP_\theta/d\nu ]$ for every $\theta\in\Theta$.

Denote by $\tau _{\mathbb{R}^k}$ the standard topology on $\mathbb{R}^k$.

For each subset $\Theta'\subseteq\Theta $, define $\Delta (\Theta '):=\{\eta (\theta ):\theta\in\Theta '\}$. Since $\Delta (\Theta )\subseteq \mathbb{R}^k$, we can conclude that $(\Delta (\Theta ), \tau _{\mathbb{R}^k}|_{\Delta (\Theta )})$ is a separable metrizable space, which implies that there exists a non-empty countable subset $\Delta _0\subseteq\Delta (\Theta )$ that is dense in $(\Delta (\Theta ), \tau _{\mathbb{R}^k}|_{\Delta (\Theta )})$, where $\tau _{\mathbb{R}^k}|_{\Delta (\Theta )}:=\{B\cap\Delta (\Theta ):B\in \tau _{\mathbb{R}^k}\}$ is the subspace topology.

Since $\Delta _0\subseteq\Delta (\Theta )$, we can conclude that there exists a non-empty countable subset $\Theta _0\subseteq\Theta$ such that $\Delta _0=\Delta (\Theta _0)$.

Define $D(x,\Theta _0):=\big\{y\in \mathcal{X}:(\exists h_{xy}\in (0,\infty ))(\forall \theta \in \Theta _0)\big(f_\theta (y)=f_\theta (x)h_{xy}\big)\big\}$.

Let $x,y\in\mathcal{X}$ be arbitrary elements such that $y\in D(x,\Theta _0)$. Next, we will prove that $T(x)=T(y)$. Then, there exists $h_{xy}\in (0,\infty )$ such that $f_\theta (y)=f_\theta (x)h_{xy}$ for every $\theta\in\Theta_0$, allowing us to infer, using the definition of $f_\theta$, that, for each $\theta\in\Theta _0$, we have the constant $e^{\langle \eta (\theta ),T(y)-T(x)\rangle }=h_{xy}$ and, consequently, that 

\begin{equation}
\label{eqmeth:3.3.2}
\langle \eta (\theta ),T(y)-T(x)\rangle =c
\end{equation} for every $\theta\in \Theta _0$, where $c:=\log (h_{xy})$.

Since $\Delta (\Theta _0)$ is dense in $(\Delta (\Theta ),\tau _{\mathbb{R}^k}|_{\Delta (\Theta )})$, we can conclude, using (\ref{eqmeth:3.3.2}) and the continuity of the function $\mathbb{R}^k\to \mathbb{R},\,t\mapsto \langle t,T(y)-T(x)\rangle $, that $\langle \eta (\theta ),T(y)-T(x)\rangle =c$ for every $\theta\in\Theta$.

Therefore, using the definition of the inner product $\langle \cdot ,\cdot \rangle $, we conclude that, for each $\theta\in\Theta$, we have $\sum _{i=1}^k\eta _i(\theta )(T_i(y)-T_i(x))=c$, which implies, using the hypothesis, that $T_i(y)-T_i(x)=0$ for every $i\in \{1,\cdots,k\}$ and, consequently, that $T_i(x)=T_i(y)$ for every $i\in \{1,\cdots, k\}$. Hence, we obtain $T(x)=T(y)$.

From the previous results, we conclude that $T$ is sufficient and that there exists a non-empty countable subset $\Theta_0\subseteq\Theta$ such that, for any $x,y\in\mathcal{X}$ satisfying $y\in D(x,\Theta _0)$, we have $T(x)=T(y)$, which implies, using Method \ref{meth:3.1}, that $T$ is minimal sufficient.

\end{proof}

\section{Acknowledgments}

The first author acknowledges the financial support provided by CAPES (Coordenação de Aperfeiçoamento de Pessoal de Nível Superior), Brazil, and by CNPq (Conselho Nacional de Desenvolvimento Científico e Tecnológico), Brazil, Process No. 130535/2024-9.

\appendix
\section{Measurable-Space Background}
\label{sec:definitions}

For convenience, we collect in this appendix the measurable-space background used in Sections \ref{sec:methods} and \ref{sec:proofs}.

\begin{definition}\label{def:statistical_model}
A statistical model is a triple $(\mathcal{X},\Sigma,\{P_\theta\}_{\theta\in\Theta })$, where $(\mathcal{X},\Sigma )$ is a measurable space, $\Theta$ is a non-empty set, and $\{P_\theta\}_{\theta\in \Theta}$ is a collection of probability measures on $\mathcal{X}$.
\end{definition}

In practice, this model is often specified implicitly through a random sample, which is a measurable function $X:\Omega\to \mathcal{X}$ defined on a probability space such that the distribution induced by $X$ is $P_\theta$ for some $\theta\in\Theta$. A canonical example arises from a sequence of $n$ independent and identically distributed (i.i.d.) random variables, $X_1,\cdots, X_n$, where each variable shares a common distribution law $Q_\theta $ for some $\theta\in\Theta$ and takes values in a measurable space $(\mathcal{Y},\Sigma_\mathcal{Y} )$. This sequence represents the random sample $X:=(X_1,\cdots,X_n)$, whose probability law is the product measure $\otimes _{i=1}^nQ_\theta$, since $X_1,\cdots, X_n$ are independent. Consequently, the statistical model associated with this random sample is $(\mathcal{X},\Sigma,\{P_\theta\}_{\theta\in\Theta }):=(\mathcal{Y}^n,\otimes _{i=1}^n\Sigma_\mathcal{Y} ,\{\otimes _{i=1}^nQ_\theta\}_{\theta\in\Theta} )$, where $\otimes _{i=1}^n\Sigma _\mathcal{Y}$ is the product $\sigma$-algebra. For brevity, a sequence of i.i.d. random variables $X_1,\cdots,X_n$ itself is also commonly referred to as a random sample.

\begin{definition}\label{def:suff}
Let $\mathcal{E}:=(\mathcal{X},\Sigma,\{P_\theta\}_{\theta\in\Theta })$ be a statistical model and $(\mathcal{T},\Sigma _\mathcal{T})$ a measurable space. A statistic, i.e., a measurable function $T:\mathcal{X}\to \mathcal{T}$ is said to be ($\mathcal{E}$-)sufficient (w.r.t. that measurable space) iff for each $E\in \Sigma $, there exists a measurable function $\kappa _E:\mathcal{T}\to \mathbb{R}$ such that, for each $\theta\in\Theta$, we have $P_\theta (E|T)=\kappa _E(T)$ $P_\theta$-a.e. Furthermore, $T$ is said to be minimal ($\mathcal{E}$-)sufficient (w.r.t. measurable space) iff it is sufficient and, given any measurable space $(\mathcal{S},\Sigma _\mathcal{S})$ and sufficient statistic $S:\mathcal{X}\to (\mathcal{S},\Sigma _\mathcal{S})$, there exists a measurable function $f:\mathcal{S}\to \mathcal{T}$ such that, for each $\theta\in\Theta$, we have $T=f(S)$ $P_\theta$-a.e.
\end{definition}

Informally, $T$ is sufficient iff the conditional probability $P_\theta (E|T)$ does not depend on $\theta$ for every $E\in \Sigma$ and $T$ is minimal sufficient iff it is sufficient and a function of any other sufficient statistic.

\begin{definition}\label{def:stand_Borel}
A standard Borel space is a measurable space whose $\sigma$-algebra is the Borel $\sigma$-algebra of a complete separable metric space.
\end{definition}

Examples of standard Borel spaces include the Euclidean space $\mathbb{R}^n$, the extended real line $\overline{\mathbb{R}}:=\mathbb{R}\cup \{\pm \infty \}$, and the space of all real sequences $\mathbb{R}^\mathbb{N}$, which is often denoted by $\mathbb{R}^\infty $. Furthermore, if $(\mathcal{X},\Sigma )$ is a standard Borel space and $B\in \Sigma $, then $(B,\Sigma |_B)$ is also a standard Borel space, where $\Sigma |_B:=\{E\cap B:E\in \Sigma \}$ is the trace $\sigma$-algebra. This is a consequence of Corollary 13.4 in \cite{Kechris1995}.

\begin{definition}\label{def:countable_space}
A measurable space $(\mathcal{X},\Sigma )$ is said to be countably generated iff $\Sigma$ is generated by a countable collection of measurable sets.
\end{definition}

Examples of such spaces include the Borel $\sigma$-algebras on second-countable topological spaces such as $\mathbb{R}^n$, $\overline{\mathbb{R}}$, and $\mathbb{R}^\mathbb{N}$. Moreover, every standard Borel space is countably generated. This is a consequence of Proposition 2.1.9 in \cite{Srivastava1998}.

Finally, we introduce two more concepts.

\begin{definition}\label{def:separable}
A measurable space $(\mathcal{X},\Sigma )$ is said to be separable iff it is countably generated and $\{x\}\in \Sigma $ for every $x\in \mathcal{X}$.
\end{definition}

Every standard Borel space is separable since in a metric space every singleton is a closed set and, hence, measurable. Moreover, it can be proved that a measurable space is separable iff it is the Borel $\sigma$-algebra of a separable metric space, as can be seen in Proposition 4.3.10 of \cite{Doberkat2015}.

\begin{definition}\label{def:analytic_Borel}
A measurable space $(\mathcal{X},\Sigma )$ is said to be analytic Borel iff it is separable and there exist a standard Borel space $(\mathcal{Y},\Sigma _\mathcal{Y})$ and a surjective measurable function $f:\mathcal{Y}\to \mathcal{X}$.
\end{definition}

Hence, every standard Borel space is an analytic Borel space.

\end{document}